\documentclass[11pt]{amsart}
\usepackage{fullpage}
\usepackage{amsfonts}
\usepackage{amssymb}
\usepackage{amsthm}
\usepackage{amsmath}
\usepackage{aliascnt}
\usepackage{hyperref}

\allowdisplaybreaks
\numberwithin{equation}{section}

\newaliascnt{thm}{equation}
\newtheorem{thm}[thm]{Theorem}
\aliascntresetthe{thm}
\newaliascnt{prop}{equation}
\newtheorem{prop}[prop]{Proposition}
\aliascntresetthe{prop}
\newaliascnt{lem}{equation}
\newtheorem{lem}[lem]{Lemma}
\aliascntresetthe{lem}
\newaliascnt{cor}{equation}
\newtheorem{cor}[cor]{Corollary}
\aliascntresetthe{cor}
\newaliascnt{rmk}{equation}
\newtheorem{rmk}[rmk]{Remark}
\aliascntresetthe{rmk}

\theoremstyle{definition}
\newaliascnt{dfn}{equation}
\newtheorem{dfn}[dfn]{Definition}
\aliascntresetthe{dfn}
\newaliascnt{data}{equation}
\newtheorem{data}[data]{Data}
\aliascntresetthe{data}
\newaliascnt{obs}{equation}

\aliascntresetthe{obs}
\makeatletter  
\def\@endtheorem{\qed\endtrivlist\@endpefalse }
\makeatother
\newaliascnt{ex}{equation}
\newtheorem{ex}[ex]{Example}
\aliascntresetthe{ex}

\newcommand\Ac{\mathcal{A}_c}
\newcommand\Bc{\mathcal{B}_c}
\newcommand\C{\mathcal{C}}
\newcommand\J{\mathcal{J}}
\newcommand\K{\mathcal{K}}
\renewcommand\k{\Bbbk}
\newcommand\m{\mathfrak{m}}
\newcommand\N{\mathbb{N}}
\renewcommand\P{\mathbb{P}_{\k}}
\newcommand\p{\mathfrak{p}}
\newcommand\R{\mathcal{R}}
\newcommand\Sym{\mathcal{S}}
\newcommand\tr{^\mathrm{T}}
\newcommand\w{w}
\newcommand\ua{\underline{\alpha}}
\newcommand\ub{\underline{\beta}}
\newcommand\us{\underline{\sigma}}
\newcommand\ut{\underline{\w}}
\newcommand\as{\ua\cdot\us}
\newcommand\bs{\ub\cdot\us}

\DeclareMathOperator\ann{ann}
\DeclareMathOperator\Ass{Ass}
\DeclareMathOperator\coker{coker}
\DeclareMathOperator\depth{depth}
\DeclareMathOperator\Ext{Ext}
\DeclareMathOperator\Fitt{Fitt}
\DeclareMathOperator\Hom{Hom}
\DeclareMathOperator\Ht{ht}
\DeclareMathOperator\im{im}
\DeclareMathOperator\grade{grade}
\DeclareMathOperator\pd{pd}
\DeclareMathOperator\Quot{Quot}
\DeclareMathOperator\rank{rank}
\DeclareMathOperator\reg{reg}
\DeclareMathOperator\Scr{Scr}
\DeclareMathOperator\Spec{Spec}
\DeclareMathOperator\Supp{Supp}
\DeclareMathOperator\syl{syl}

\author{Jeff Madsen}
\address{Department of Mathematics, University of Notre Dame, Notre Dame, IN 46556}
\email{jmadsen@nd.edu}
\title{Equations of Rees algebras of ideals in two variables}
\date{\today}

\begin{document}

\begin{abstract}
Let $I$ be an ideal of height two in $R=\k[x_0,x_1]$ generated by forms of the same degree, and let $\K$ be the ideal of defining equations of the Rees algebra of $I$. Suppose that the second largest column degree in the syzygy matrix of $I$ is $e$. We give an algorithm for computing the minimal generators of $\K$ whose degree is at least $e$, as well as a simple formula for the bidegrees of these generators. In the case where $I$ is an almost complete intersection, we give a generating set for each graded piece $\K_{i,*}$ with $i\ge e-1$.
\end{abstract}

\maketitle

Let $I$ be an ideal in a Noetherian ring $R$. The Rees algebra of $I$, defined as $\R(I)=R[It]=\bigoplus_{j=0}^\infty I^jt^j$, is a central object in both commutative algebra, where it is used to study the asymptotic behavior of $I$, and algebraic geometry, where it corresponds to blowing up the closed subscheme $V(I)$ in $\Spec R$. A common approach to understanding Rees algebras is by representing them as quotients of polynomial rings. Namely, if $I=(f_1,\ldots,f_n)$, then $\R(I)$ is a quotient of the polynomial ring $S=R[T_1,\ldots,T_n]$ via the $R$-linear map $T_i\mapsto f_it$. The kernel of this map is an ideal $\K\subset S$, called the ideal of defining equations of $\R(I)$. Much work has been done to determine the ideal $\K$ under various conditions (such as \cite{bm15,bcs10,cd13a,hsv12,hsv13,joh97,lan14,lp14,mor96,mu96,tai02,uv93}). However, even in the relatively simple case where $R=\k[x_0,x_1]$ is a polynomial ring in two variables over a field and $I$ is a homogeneous ideal generated in a single degree $d$, the problem of determining the equations of $\R(I)$ remains open.

This problem is of interest not only in commutative algebra and algebraic geometry, but also in elimination theory and geometric modeling, where it is connected to implicitization of parametrized varieties (e.g.\ \cite{bj03,cox08,csc98,sc95,sgd97}). We focus on the case $R=\k[x_0,x_1]$ where $\k$ is an algebraically closed field. To any height two homogeneous ideal $I\subset R$ generated by forms $f_1,\ldots,f_n$ of degree $d$, we may associate a morphism $\Phi:\P^1\xrightarrow{(f_1:\cdots:f_n)}\P^{n-1}$, which parametrizes a curve $\C\subset\P^{n-1}$. Conversely, to any map $\Phi:\P^1\to \C$ parametrizing a curve in $\P^{n-1}$ by equations $(f_1:\cdots:f_n)$, we may associate the ideal $I=(f_1,\ldots,f_n)\subset R$. Then the Rees algebra $\R(I)$ is the bihomogeneous coordinate ring of the graph of $\Phi$,
\[\Gamma_\Phi=\{(p,\Phi(p))\mid p\in\P^1\}\subset\P^1\times\P^{n-1}.\]
Since $S=R[T_1,\ldots,T_n]$ is the bihomogeneous coordinate ring of $\P^1\times\P^{n-1}$, this means that the ideal $\K\subset S$ of defining equations of $\R(I)$ is the ideal of $\Gamma_\Phi$. That is, for any bihomogeneous $h\in S$,
\[h\in\K\Leftrightarrow h(x_0,x_1,f_1(x_0,x_1),\ldots,f_n(x_0,x_1))=0.\]
In particular, the elements of $\K$ having $x$-degree 0 are just the implicit equations of the curve $\C$. (By $x$-degree we mean the degree in the variables $x_0,x_1$, as opposed to the degree in $T_1,\ldots,T_n$, or $T$-degree.)

The first step in computing the equations of the Rees algebra is to compute the syzygy matrix $\varphi$ of $I$, which determines the $T$-linear part of $\K$. Since $I$ is generated in one degree, each column of the matrix $\varphi$ has the same degree. Several results have been proved by imposing restrictions on the syzygy matrix. For ideals generated by three forms, Cox, Hoffman, and Wang \cite{chw08} gave a recursive algorithm, using Sylvester forms, for computing a minimal generating set for the equations of the Rees algebra when $I$ has a syzygy of degree 1. This was generalized by Cortadellas and D'Andrea \cite{cd13b} to the case where a syzygy of $I$ of minimal degree $\mu$, when viewed as a column vector, has linearly dependent entries (a ``generalized zero''). However, their method only produces part of a minimal generating set. In particular, the equations produced by this algorithm all have $x$-degree $\ge\mu-1$. These are in fact all the minimal generators of $\K$ having $x$-degree $\ge\mu-1$, as proved by Kustin, Polini, and Ulrich \cite{kpu13}.

Kustin, Polini, and Ulrich \cite{kpu11} took another approach to determine the equations of $\K$ when $I$ has any number of generators and has a presentation matrix $\varphi$ which is almost linear, meaning all but one of the columns of $\varphi$ have linear entries. They showed that the Rees algebra is a quotient of a ring $A$, which is the coordinate ring of a rational normal scroll, by a height one prime ideal $\K A$. The equations of the scroll being easy to compute, this reduces the problem to computing the generators of the divisor $\K A$ on $A$, which was carried out in \cite{kpu09}. The equations are even given explicitly, assuming $\varphi$ is put into a canonical form.

In this paper, we extend the method of Kustin, Polini, and Ulrich to arbitrary height two ideals $I$ in $\k[x_0,x_1]$ generated by forms of the same degree. Unlike the almost linearly presented case, we cannot produce all the defining equations of the Rees algebra in this way, but we do obtain all the minimal generators of $x$-degree greater than or equal to the second largest degree of a column of $\varphi$, and in some cases, in degree one less. \autoref{kregen} gives an abstract form of these generators, and \autoref{kregenalg} gives a recursive algorithm for computing them. In particular, their bidegrees follow a simple formula; see \autoref{kregencor}.

In \autoref{exalp}, we show how to use \autoref{kregen} to solve the almost linearly presented case, reproving \cite[Theorem 3.6]{kpu11}. In \autoref{exgen0}, we show how the algorithm in \autoref{kregenalg} reduces to the Sylvester form algorithm in \cite[Theorem 2.10]{cd13b}, when $I$ has three generators and has a syzygy of minimal degree with a generalized zero. In \autoref{ex2n}, we give another example, applying \autoref{kregenalg} to the case where $I$ has three generators and a syzygy of minimal degree 2 without a generalized zero, a case which was previously solved in \cite{bus09} and \cite{cd13b} using Morley forms.

In Section 1, we introduce the setup of the problem and the necessary definitions. In particular, we construct a sequence of modules $E_m$ which are successive approximations of $I$.

In Section 2, we determine the structure of the modules $\R(E_m)$ by embedding each $E_m$ in a free module. This allows us to embed $\R(E_m)$ in a ring of the form $\R(M)$ where $M=\bigoplus_{i=1}^s\m^{\sigma_i}(\sigma_i)$, which we show in \autoref{rns} is the coordinate ring of a rational normal scroll. This embedding is actually an integral extension, and the two rings are equal in large degree, as we show in \autoref{remint}.

In Section 3, we use the results of Section 2 to compute the ideal of equations of $\R(E_{m+1})$ in $\R(E_m)$ by first computing it in $\R(M)$ and then contracting to $\R(E_m)$. In \autoref{kregen}, we give the elements in $\R(M)$ which contract to a minimal generating set of the $x$-degree $\ge d_m$ part of the ideal of equations of $\R(E_{m+1})$. The contraction may be made explicit using the algorithm given in \autoref{kregenalg}. \autoref{kregencor} applies this in particular to the case $m=n-2$ to obtain the bidegrees of the minimal generators of $\K$ with $x$-degree $\ge d_{n-2}$.

In Section 4, we consider specifically plane curves, that is, the case $n=3$. In this case, we can be more explicit about the $\R(E_1)$-module structure of $\R(M)$, in \autoref{rffree}. We use this to give a generating set, in general not minimal, for each graded component $\K_{i,*}$ with $i\ge d_1-1$.

\section{Setup}

Throughout this paper, we assume $\k$ is a field, $R$ is the two-variable polynomial ring $R=\k[x_0,x_1]$, and $\m$ is the homogeneous maximal ideal $\m=(x_0,x_1)$. We use the convention that $\N$ is the set of nonnegative integers.

Although we wish to compute Rees algebras of ideals, it will be helpful to use Rees algebras of more general modules. If $E$ is a finitely generated $R$-module, then following \cite{suv03}, we define the Rees algebra $\R(E)$ to be the quotient of the symmetric algebra $\Sym(E)$ by its $R$-torsion. As $R$ is a domain, this is consistent with the other standard definitions of Rees algebras of modules, such as in \cite{ehu03}, and it is also consistent with the usual definition of the Rees algebra of an ideal. Recall that if $\varphi$ is the $m\times n$ presentation matrix of $E$, then $\Sym(E)\cong R[T_1,\ldots,T_n]/(g_1,\ldots,g_m)$ where
\begin{equation}\label{eq:defg}\begin{bmatrix}g_1&\cdots&g_m\end{bmatrix}=\begin{bmatrix}T_1&\cdots&T_n\end{bmatrix}\varphi.\end{equation}
The Krull dimension of the Rees algebra is $\dim\R(E)=\dim R+\rank E=2+\rank E$ (\cite[Proposition~2.2]{suv03}).

Let $S=R[T_1,\ldots,T_n]$, with $\deg x_i=(1,0)$ and $\deg T_i=(0,1)$. We refer to the first degree in this bigrading as the $x$-degree or $\deg_x$, and the second degree as the $T$-degree or $\deg_T$. By $S_{i,j}$ we mean the bidegree $(i,j)$ part of $S$, viewed as a $\k$-vector space. Write $S_{i,*}$ for the $k[T_1,\ldots,T_n]$-module $\bigoplus_jS_{i,j}$, and $S_{*,j}$ for the $R$-module $\bigoplus_iS_{i,j}$.

We shall use the same bigrading on any quotient of $S$, such as $\R(I)$. Note that this is not the same as the bigrading induced by the inclusion $\R(I)=R[It]\subset R[t]$. For example, the element $f_it$ has bidegree $(d,1)$ in $R[t]$, however we will think of it as having bidegree $(0,1)$ since it corresponds to the element $T_i\in S$. Thus to be more precise, we consider $\R(I(d))$ rather than $\R(I)$, so that the grading induced by $S$ is the natural grading on the Rees algebra.

If $I$ is a height two ideal of $R$ minimally generated by $n$ forms of degree $d$, then $I(d)$ has a Hilbert-Burch resolution
\begin{equation}\label{eq:ires}0\to\bigoplus_{i=1}^{n-1}R(-d_i)\xrightarrow{\varphi}R^n\to I(d)\to0\end{equation}
where $I$ is generated by the maximal minors of $\varphi$, and thus $\sum_{i=1}^{n-1}d_i=d$. We may assume the column degrees of $\varphi$ are nondecreasing, that is, $1\le d_1\le\cdots\le d_{n-1}$.

For $1\le m\le n-1$, let $\varphi_m$ be the $n\times m$ matrix consisting of the first $m$ columns of $\varphi$ and let $E_m=\coker\varphi_m$, so that $\varphi_{n-1}=\varphi$ and $E_{n-1}=I(d)$. There is a sequence of surjections
\[R^3\twoheadrightarrow E_1\twoheadrightarrow E_2\twoheadrightarrow\cdots\twoheadrightarrow E_{n-1}=I(d)\]
which induce surjections
\[S\twoheadrightarrow\R(E_1)\twoheadrightarrow\R(E_2)\twoheadrightarrow\cdots\twoheadrightarrow\R(E_{n-1})=\R(I(d)).\]
We think of the rings $\R(E_m)$ as successive approximations of $\R(I(d))$. Instead of directly computing the equations of $\R(I(d))$ in $S$, we may compute the equations in $\R(E_{n-2})$, where it will be simpler. We then need to compute the equations of $\R(E_{n-2})$ in $\R(E_{n-3})$, and so on. That is, we have a sequence of inclusions
\[0\subsetneq\K_1\subsetneq\K_2\subsetneq\cdots\subsetneq\K_{n-1}=\K\]
where each $\K_m$ is a prime ideal with $\Ht\K_m=\dim S-\dim\R(E_m)=(2+n)-(2+\rank E_m)=m$. We study the modules $\K_{m+1}\R(E_m)\cong\K_{m+1}/\K_m$. A generating set for $\K$ may be obtained by combining the lifts of generating sets of all the modules $\K_{m+1}/\K_m$.

We summarize the data that we will be assuming.

\begin{data}\label{idealdata}
Let $I$ be a height 2 ideal in $R$ generated by $n$ forms $f_1,\ldots,f_n$ of degree $d$, with presentation matrix $\varphi$, and resolution as given in \eqref{eq:ires}. Let $S=\k[T_1,\ldots,T_n]$, let $g_1,\ldots,g_{n-1}\in S$ be the equations of the symmetric algebra of $I$ as in \eqref{eq:defg}, and let $\K\subset S$ be the ideal of equations of the Rees algebra of $I$. For $1\le m\le n-1$, let $\varphi_m$ be the $n\times m$ matrix consisting of the first $m$ columns of the presentation matrix $\varphi$ and let $E_m=\coker\varphi_m$. Let $\K_m\subset S$ be the ideal of equations of the Rees algebra of $E_m$.
\end{data}

By \eqref{eq:defg}, the equations of $\Sym(E_m)$ are $g_1,\ldots,g_m$. Thus for each $m$, there is a short exact sequence
\begin{equation}\label{eq:symrees}0\to\K_m/(g_1,\ldots,g_m)\to\Sym(E_m)\to\R(E_m)\to0.\end{equation}
Let us deduce some simple properties of the modules $E_m$.

\begin{lem}\label{efree}
Assume \autoref{idealdata} and let $1\le m\le n-1$. Then $E_m$ is torsionfree of rank $n-m$, and for any $w\in\m$, $(E_m)_w$ is a free $R_w$-module.
\end{lem}

\begin{proof}
Because $\varphi$ is injective, and $\varphi_m$ is the restriction of $\varphi$ to a submodule, $\varphi_m$ is injective. Thus $E_m$ has a free resolution
\begin{equation}\label{eq:resem}0\to\bigoplus_{i=1}^mR(-d_i)\xrightarrow{\varphi_m}R^n\to E_m\to0.\end{equation}
By expanding the determinants along the last $n-1-m$ columns, we see that $I=I_{n-1}(\varphi)\subset I_m(\varphi_m)$. Because $\Ht I=2$, we get $\Ht I_m(\varphi_m)=2$. So if we invert $w\in\m$, we get $I_m(\varphi_m)_w=R_w$. This means that after inverting $w$, \eqref{eq:resem} splits, yielding $(E_m)_w\cong R_w^{m-n}$.

It remains to show that $E_m$ is torsionfree, or equivalently that if $\p\neq0$ is a homogeneous prime ideal of $R$, then $\p\not\in\Ass_R(E_m)$. If $\p\neq\m$, then $(E_m)_\p$ is free by the above, so $\p R_\p\not\in\Ass_{R_\p}((E_m)_\p)$ and thus $\p\not\in\Ass_R(E_m)$. On the other hand, $\pd E_m\le 1$ by \eqref{eq:resem}, so $\depth E_m\ge 1$ by the Auslander-Buchsbaum formula. Hence $\m\not\in\Ass_R(E_m)$, therefore $E_m$ is torsionfree.
\end{proof}

As a consequence, we obtain a formula for $\K_m$. While this description does not provide any information about the generators of $\K_m$, it is the basis upon which we will perform our calculations involving $\K_m$.

\begin{prop}\label{resat}
Assume \autoref{idealdata} and let $1\le m\le n-1$. Then $\K_m=(g_1,\ldots,g_m):_S\m^\infty$.
\end{prop}

\begin{proof}
By \autoref{efree}, for any $w\in\m$, $(E_m)_w$ is a free $R_w$-module. Therefore $\R(E_m)_w=\R((E_m)_w)=\Sym((E_m)_w)=\Sym(E_m)_w$. It follows from \eqref{eq:symrees} that $\K_mS_w=(g_1,\ldots,g_m)S_w$, which means $\K_m:_S\m^\infty=(g_1,\ldots,g_m):_S\m^\infty$. But $\K_m$ is a prime ideal, so $\K_m=\K_m:_S\m^\infty=(g_1,\ldots,g_m):_S\m^\infty$.
\end{proof}

\section{The modules \texorpdfstring{$\R(E_m)$}{R(E\_m)}}

Fix $m$. To determine the Rees algebra of $E_m$, we begin by embedding $E_m$ in a free module. Namely, we show in \autoref{reflfree} that $F=E_m^{**}$ is free (where $-^*=\Hom_R(-,R)$); we have an injection $E_m\hookrightarrow E_m^{**}$ since $E_m$ is torsionfree by \autoref{efree}. The free module $F$ is generated in nonpositive degrees, while $E_m$ is generated in degree 0. Thus $E$ is a submodule of $M=F_{\ge0}$. This section is concerned with computing the Rees algebras of $F$ and $M$ and comparing them to the Rees algebra of $E_m$.

The lemma below is the graded version of \cite[Proposition~2.1]{kod95}.

\begin{lem}\label{reflfree}
Let $E$ be an $R$-module of rank $s$ generated in degree 0. Then $E^{**}\cong\bigoplus_{i=1}^sR(\sigma_i)$ with all $\sigma_i\ge0$.
\end{lem}

\begin{proof}
Set $F=E^{**}$. Then $F$ is reflexive, so it satisfies Serre's condition $S_2$. Since $F$ is torsionfree, $\dim F=\dim R=2$, hence $F$ is maximal Cohen-Macaulay. Because $R$ is regular, the Auslander-Buchsbaum formula implies $F$ is free. Thus we may write $F=\bigoplus_{i=1}^sR(\sigma_i)$.

To see that all $\sigma_i\ge0$, choose a surjection $R^n\twoheadrightarrow E$ and consider the map $\xi:R^n\to F$ obtained by composing this surjection with the natural map $\theta:E\to E^{**}=F$. Assume by way of contradiction that some $\sigma_i$, say $\sigma_s$, is negative. Then, viewing $\xi$ as an $s\times n$ matrix, the last row of $\xi$ must be zero. This means $\im\xi\subset\bigoplus_{i=1}^{s-1}R(\sigma_i)$, so $\rank(\im\xi)<s$. But $\im\xi=\im\theta$, so $\rank(\im\xi)=\rank E=s$. By contradiction, all $\sigma_i$ are nonnegative.
\end{proof}

Recall from \autoref{efree} that $E_m$ is a torsionfree $R$-module of rank $n-m$. Using \autoref{reflfree}, we extend \autoref{idealdata}:

\begin{data}\label{reesdata}
Assume \autoref{idealdata}, fix $1\le m\le n-1$, and set $s=n-m$. Let $E=E_m$ and let $F=E^{**}\cong\bigoplus_{i=1}^sR(\sigma_i)$, with $\sigma_1\ge\cdots\ge\sigma_s\ge0$. Let $r$ be the unique integer with $1\le r\le s$ such that $\sigma_1\ge\cdots\ge\sigma_r>0=\sigma_{r+1}=\cdots=\sigma_s$. Let $M=F_{\ge0}=\bigoplus_{i=1}^s\m^{\sigma_i}(\sigma_i)$. Let $\xi$ be the composite map $R^n\twoheadrightarrow E\hookrightarrow F$.
\end{data}

Because $E$ is generated in degree 0, the inclusion $E\hookrightarrow F$ factors as $E\hookrightarrow M\hookrightarrow F$. Therefore there are inclusions of Rees algebras $\R(E)\hookrightarrow\R(M)\hookrightarrow\R(F)$. This will allow us to do computations in $\R(E)$ by extending to $\R(M)$ or $\R(F)$ and then contracting, as in \autoref{kre}. The Rees algebra of $F$ is easy to understand.

\begin{rmk}\label{fpoly}
Assume \autoref{reesdata}. Then $\R(F)=R[\w_1,\ldots,\w_s]=\k[x_0,x_1,\w_1,\ldots,\w_s]$, where $\deg x_i=(1,0)$ and $\deg\w_i=(-\sigma_i,1)$. Viewing $\R(E)$ as a subset of $\R(F)$, we have
\[\begin{bmatrix}T_1&\cdots&T_n\end{bmatrix}=\begin{bmatrix}\w_1&\cdots&\w_s\end{bmatrix}\xi.\]
\end{rmk}

\begin{proof}
Since $F=\bigoplus_{i=1}^sR(\sigma_i)$ is a free module, the Rees algebra $\R(F)$ is the same as the symmetric algebra $\Sym(F)$, which is a polynomial ring $R[\w_1,\ldots,\w_s]$. Each variable $\w_i$ corresponds to one generator of $F$. Since the $i$th generator of $F$ has degree $-\sigma_i$, the variable $\w_i$ has bidegree $(-\sigma_i,1)$. The $R$-module map $\xi:R^n\twoheadrightarrow E\hookrightarrow F$ induces the $R$-algebra map $\R(\xi):S=\R(R^n)\twoheadrightarrow\R(E)\hookrightarrow\R(F)$ which lets us view $\R(E)$ as a submodule of $\R(F)$. Therefore
\[\begin{bmatrix}T_1&\cdots&T_n\end{bmatrix}=\begin{bmatrix}\w_1&\cdots&\w_s\end{bmatrix}\xi.\qedhere\]
\end{proof}

The Rees algebra of $M$ is also not difficult to describe. First, will to define the $R$-algebra $\Scr(\us)$, which we will then show is isomorphic to $\R(M)$ in \autoref{rns}. The ring $\Scr(\us)$ is the coordinate ring of a rational normal scroll if all the $\sigma_i$ are positive; otherwise, it is a cone over a rational normal scroll. In any case, it is a Cohen-Macaulay normal domain of dimension $s+2$ (see \cite{eis05} or \cite{eh87}).

\begin{dfn}\label{dfscroll}
Let $\sigma_1,\ldots,\sigma_s$ be nonnegative integers, and let $V=R[\{v_{i,j}\mid 1\le i\le s,\ 0\le j\le\sigma_i\}]$. Let $\Gamma$ be the matrix with entries in $V$:
\[\Gamma=\left[\begin{array}{c|cccc|c|cccc}x_0&v_{1,0}&v_{1,1}&\cdots&v_{1,\sigma_1-1}&\cdots&v_{s,0}&v_{s,1}&\cdots&v_{s,\sigma_s-1}\\x_1&v_{1,1}&v_{1,2}&\cdots&v_{1,\sigma_1}&\cdots&v_{s,1}&v_{s,2}&\cdots&v_{s,\sigma_s}\end{array}\right].\]
Define the ring $\Scr(\us)=\Scr(\sigma_1,\ldots,\sigma_s)=V/I_2(\Gamma)$.
\end{dfn}

\begin{prop}\label{rns}
Assume \autoref{reesdata}. Then
\[\R(M)=R[\{x_0^{\sigma_i-j}x_1^j\w_i\mid 1\le i\le s,\ 0\le j\le\sigma_i\}]\subset R[\w_1,\ldots,\w_s]=\R(F)\]
and $\R(M)\cong\Scr(\us)$.
\end{prop}

\begin{proof}
From \autoref{fpoly}, $\R(F)=R[\w_1,\ldots,\w_s]$ where $\deg x_i=(1,0)$ and $\deg\w_i=(-\sigma_i,1)$. Thus the monomials in $\R(F)$ of $x$-degree 0 are $x_0^{\sigma_i-j}x_1^j\w_i$, so \[\R(M)=\R(F)_{\ge0,*}=R[\{x_0^{\sigma_i-j}x_1^j\w_i\mid 1\le i\le s,\ 0\le j\le\sigma_i\}].\]

Now define $V$ and $\Gamma$ as in \autoref{dfscroll}, and define a surjective $R$-linear map $\theta:V\twoheadrightarrow\R(M)$ by $v_{i,j}\mapsto x_0^{\sigma_i-j}x_1^j\w_i$. It is easy to see that $I_2(\Gamma)\subset\ker\theta$, so $\theta$ descends to a map $\overline{\theta}:\Scr(\us)=V/I_2(\Gamma)\twoheadrightarrow\R(M)$. Since $\dim\Scr(\us)=s+2=2+\rank M=\dim\R(M)$, $\overline{\theta}$ is a surjective map between domains of the same dimension, hence it is an isomorphism.
\end{proof}

Because $\R(F)$ and $\R(M)$ have straightforward descriptions, we want to compare $\R(E)$ to these rings. The first step is to compare the modules $E$ and $F$ by looking at $F/E$. In the remark below, we see that $F/E$ is Artinian and is zero in degree $\ge d_m-1$. This will be used in \autoref{remint} to show that $\R(F)$ and $\R(E)$ agree in these degrees. As another consequence, the resolution \eqref{eq:res1} gives an easy way of computing the matrix $\xi$, as the transpose of the syzygy matrix of $\varphi_m\tr$. We need to know $\xi$ in order to make the homomorphism $\R(E)\hookrightarrow\R(F)$ explicit (\autoref{fpoly}).

\begin{lem}\label{fmode}
Assume \autoref{reesdata}.
\begin{itemize}
\item[(a)] $F/E$ has free resolution
\begin{equation}\label{eq:res2}0\to\bigoplus_{i=1}^mR(-d_i)\xrightarrow{\varphi_m}R^n\xrightarrow{\xi}F=\bigoplus_{i=1}^sR(\sigma_i)\to F/E\to0.\end{equation}
\item[(b)] $F/E$ is Artinian.
\item[(c)] $(F/E)_{\ge d_m-1}=0$.
\item[(d)] $\Ext^2_R(F/E,R)$ has free resolution
\begin{equation}\label{eq:res1}0\to\bigoplus_{i=1}^sR(-\sigma_i)\xrightarrow{\xi\tr}R^n\xrightarrow{\varphi_m\tr}\bigoplus_{i=1}^mR(d_i)\to\Ext^2_R(F/E,R)\to0.\end{equation}
\item[(e)] $s=n-m$ and $\sum_{i=1}^s\sigma_i=\sum_{i=1}^md_i$.
\end{itemize}
\end{lem}

\begin{proof}
(a) This resolution is obtained by combining the short exact sequence \eqref{eq:resem} with the short exact sequence
\[0\to E\to F\to F/E\to 0.\]

(b) If $\p\subset R$ is a prime ideal with $\p\neq\m$, then $E_\p$ is free by \autoref{efree}. Therefore $F_\p=E_\p^{**}=E_\p$, so $(F/E)_\p=0$. Hence $\Supp(F/E)\subset\{\m\}$, meaning $F/E$ is Artinian.

(c) Since $1\le d_1\le\cdots\le d_m$, it follows from \eqref{eq:res2} that $\reg(F/E)=d_m-2$. Since $F/E$ is Artinian, this means that $(F/E)_{\ge d_m-1}=0$.

(d) Since $F/E$ is Artinian, $\grade(F/E)=2$. By \eqref{eq:res2}, $\pd(F/E)\le2$, therefore $F/E$ is perfect of grade 2. Thus the resolution of $\Ext^2_R(F/E,R)$ is dual to the resolution of $F/E$ in \eqref{eq:res2}.

(e) From \eqref{eq:res2}, we see that the Hilbert polynomial of $F/E$ is
\begin{align*}
HP_{F/E}(t)&=\sum_{i=1}^s(t+\sigma_i+1)-n(t+1)+\sum_{i=1}^mR(t-d_i+1)\\
&=(s+m-n)t+(s+m-n+\sum_{i=1}^s\sigma_i-\sum_{i=1}^md_i).
\end{align*}
On the other hand, $F/E$ is Artinian, so its Hilbert polynomial must be zero. Therefore $s=n-m$ and $\sum_{i=1}^s\sigma_i=\sum_{i=1}^md_i$.
\end{proof}

The next theorem relates the Rees algebras $\R(E)$, $\R(M)$, and $\R(F)$. It allows us, when looking in large $x$-degrees, to pass from $\R(E)$ to $\R(M)$, where computation is easier due to the structure of $\R(M)$ given in \autoref{rns}.

\begin{thm}\label{remint}
Assume \autoref{reesdata}.
\begin{itemize}
\item[(a)] $\R(E)_{\ge d_m-1,*}=\R(M)_{\ge d_m-1,*}=\R(F)_{\ge d_m-1,*}$.
\item[(b)] $\R(M)$ is the integral closure of $\R(E)$ in its field of fractions.
\item[(c)] $E$ is a reduction of $M$ in $F$.
\end{itemize}
\end{thm}

\begin{proof}

(a) The second equality is true because $\R(M)=\R(F_{\ge0})=\R(F)_{\ge0,*}$. For the first equality, we show by induction on $j\ge0$ that $\R(E)_{i,j}=\R(M)_{i,j}$ for all $i\ge d_m-1$. If $j=0$, this is true because $\R(E)_{*,0}=R=\R(M)_{*,0}$. Let $j=1$ and let $i\ge d_m-1$. As $E$ and $M$ are torsionfree $R$-modules, $\R(E)_{*,1}=E$ and $\R(M)_{*,1}=M$. But $E_{\ge d_m-1}=F_{\ge d_m-1}=M_{\ge d_m-1}$ by \autoref{fmode}(c), therefore $\R(E)_{\ge d_m-1,1}=\R(M)_{\ge d_m-1,1}$. In particular, $\R(E)_{i,1}=\R(M)_{i,1}$ since $i\ge d_m-1$. Now suppose $j>1$ and $i\ge d_m-1$. Certainly $\R(E)_{i,j}\subset\R(M)_{i,j}$. On the other hand,
\begin{align*}
\R(M)_{i,j}&=\R(M)_{0,1}\R(M)_{i,j-1}\\
&=\R(M)_{0,1}\R(E)_{i,j-1}\text{ by induction}\\
&=\R(M)_{0,1}\R(E)_{i,0}\R(E)_{0,j-1}\\
&\subset\R(M)_{i,1}\R(E)_{0,j-1}\\
&=\R(E)_{i,1}\R(E)_{0,j-1}\text{ by the case $j=1$}\\
&=\R(E)_{i,j}.
\end{align*}
Therefore $\R(E)_{i,j}=\R(M)_{i,j}$.

(b) Since $\R(M)_{\ge d_m-1,*}=\R(E)_{\ge d_m-1,*}$ is a faithful $\R(M)$-module which is a finitely generated $\R(E)$-module, the determinant trick (\cite[Lemma~2.1.9]{sh06}) shows that the extension $\R(E)\subset\R(M)$ is integral. Since $\R(M)\cong\Scr(\us)$ by \autoref{rns}, $\R(M)$ is normal. It remains to show that $\R(E)$ and $\R(M)$ have the same field of fractions. For each $1\le i\le s$, define $p_i=x_0^{\sigma_i+d_m-1}\w_i$. Then $\deg p_i=(d_m-1,1)$, so $p_i\in\R(E)$ by (a). Also, $x_0\in R\subset\R(E)$, thus $\w_i=x_0^{-\sigma_i-d_m+1}p_i\in\Quot(\R(E))$. Since $x_0,x_1\in\Quot(\R(E))$ and $\w_i\in\Quot(\R(E))$ for all $i$, we get $\Quot(\R(E))\supset\k(x_0,x_1,\w_1,\ldots,\w_s)=\Quot(\R(F))\supset\Quot(\R(M))$, therefore $\Quot(\R(E))=\Quot(\R(M))$. Hence $\R(M)$ is the integral closure of $\R(E)$ in its field of fractions.

(c) This is equivalent to $\R(M)$ being integral over $\R(E)$ (\cite[Theorem~16.2.3]{sh06}), which we already proved in (b).
\end{proof}

\section{Equations of Rees algebras}

Now that we have some understanding of $\R(E)$, we wish to compute the ideal $\K_{m+1}\R(E)$, as discussed in Section 1. First, in \autoref{kre}, we use the fact that $\R(E)\subset\R(M)$ is an integral extension of rings to show that $\K_{m+1}\R(E)$ is the contraction of the $\R(M)$-ideal $(g_{m+1}\R(F))_{\ge0,*}=g_{m+1}\R(F)_{\ge-d_{m+1},*}$. \autoref{rfgen} gives a generating set for ideals of this form. We will put these together to obtain the $x$-degree $\ge d_m$ part of a minimal generating set of $\K_{m+1}\R(E)$ in \autoref{kregen}. While these generators are given in terms of the variables $w_1,\ldots,w_s\in\R(M)$, we desire an expression for the generators as elements of $S=R[T_1,\ldots,T_n]$. We cannot give a closed form, but in \autoref{kregenalg} we give a recursive algorithm for computing the generators of $\K_{m+1}\R(E)$ with $x$-degree $\ge d_m$.

\begin{prop}\label{kre}
Assume \autoref{reesdata}. Then \[\K_{m+1}\R(E)=(\K_{m+1}\R(M))\cap\R(E)=(g_{m+1}\R(F)_{\ge-d_{m+1},*})\cap\R(E).\]
\end{prop}

\begin{proof}
According to \autoref{remint}(b), $\R(E)\subset\R(M)$ is an integral extension of rings. Since $\K_{m+1}\R(E)$ is a prime ideal, the first equality is true by the lying over property of integral extensions.

By \autoref{resat}, $\K_m=(g_1,\ldots,g_m):_S\m^\infty$. Since $\K_m$ is the ideal of defining equations of $\R(E)$, this means $g_1,\ldots,g_m$ are zero in $\R(E)$, and thus also in $\R(M)$. \autoref{resat} also gives $\K_{m+1}=(g_1,\ldots,g_{m+1}):_S\m^\infty$. Therefore
\[\K_{m+1}\R(M)=(g_1,\ldots,g_{m+1})\R(M):_{\R(M)}\m^\infty=g_{m+1}\R(M):_{\R(M)}\m^\infty.\]
It remains to show that $g_{m+1}\R(M):_{\R(M)}\m^\infty=g_{m+1}\R(F)_{\ge-d_{m+1},*}$.

The containment ``$\supset$'' comes because \[\m^{d_{m+1}}(g_{m+1}\R(F)_{\ge-d_{m+1},*})\subset g_{m+1}\R(F)_{\ge0,*}=g_{m+1}\R(M),\] which means \[g_{m+1}\R(F)_{\ge-d_{m+1},*}\subset g_{m+1}\R(M):_{\R(M)}\m^{d_{m+1}}\subset g_{m+1}\R(M):_{\R(M)}\m^\infty.\]
For the containment ``$\subset$'', first note that \[g_{m+1}\R(M):_{\R(M)}\m^\infty\subset(g_{m+1}\R(F):_{\R(F)}\m^\infty)\cap\R(M).\]
But since $\R(F)$ is a polynomial ring, $g_{m+1}\R(F)$ is an unmixed ideal of height 1, while $\m\R(F)$ is a prime ideal of height 2, so $g_{m+1}\R(F):_{\R(F)}\m^\infty=g_{m+1}\R(F)$. Thus
\[g_{m+1}\R(M):_{\R(M)}\m^\infty\subset(g_{m+1}\R(F))\cap\R(M)=(g_{m+1}\R(F))_{\ge0,*}=g_{m+1}\R(F)_{\ge-d_{m+1},*}.\]
This shows that $g_{m+1}\R(M):_{\R(M)}\m^\infty=g_{m+1}\R(F)_{\ge-d_{m+1},*}$.
\end{proof}

According to \autoref{kre}, there are two steps to computing $\K_{m+1}\R(E)$. First, we must compute $g_{m+1}\R(F)_{\ge-d_{m+1},*}$, then we must contract to $\R(E)$. We begin with the task of determining the generators of $\R(F)_{\ge-c,*}$ as an $\R(M)$-module for any $c\ge 0$.

\subsection{Generators of truncations of \texorpdfstring{$\R(F)$}{R(F)}}

Assume \autoref{reesdata}. For any $s$-tuple $\ua=(\alpha_1,\ldots,\alpha_s)\in\N^s$, write $\ut^{\ua}=\w_1^{\alpha_1}\cdots\w_s^{\alpha_s}\in\R(F)$, and define $\as=\alpha_1\sigma_1+\cdots+\alpha_s\sigma_s$. Then since $\deg_xx_i=1$ and $\deg_x\w_i=-\sigma_i$, we may compute the $x$-degree of any monomial in $\R(F)$ with the formula $\deg_x(x_0^jx_1^k\ut^{\ua})=j+k-\as$.

In what follows, it will be convenient to use the following nonstandard notation: For $h,h'\in\R(F)$, say $h$ divides $h'$ (or write $h\mid h'$) to mean that there exists $\ell\in\R(M)$ such that $h'=\ell h$. Note that this is different from the standard notion of divisibility in $\R(F)$ because we require $\ell$ to have nonnegative degree (that is, to be in $\R(M)$ and not just $\R(F)$).

One more definition will make it easier to work with exponent vectors $\ua\in\N^s$. Recall that $r$ is the unique integer with $1\le r\le s$ such that $\sigma_i>0$ for $i\le r$ and $\sigma_i=0$ for $i>r$. In the definition below, we think of $\ua^+$ as the part of the vector corresponding to those $\sigma_i$ which are positive, and $\ua^0$ as the part corresponding to those $\sigma_i$ which are zero.

\begin{dfn}
Assume \autoref{reesdata}, and let $\ua\in\N^s$. Define $\ua^+=(\alpha_1,\ldots,\alpha_r,0,\ldots,0)\in\N^s$ and $\ua^0=(0,\ldots,0,\alpha_{r+1},\ldots,\alpha_s)\in\N^s$.
\end{dfn}

The following properties are clear from the definition.

\begin{rmk}
Assume \autoref{reesdata}, and let $\ua\in\N^s$.
\begin{itemize}
\item[(a)] $\ua=\ua^++\ua^0$.
\item[(b)] $(\ua^+)^0=0$ and $(\ua^0)^+=0$.
\item[(c)] $\as=\ua^+\cdot\us$.
\end{itemize}
\end{rmk}

The importance of distinguishing $\ua^+$ from $\ua^0$ can be seen in the following lemma, which shows that we may ignore $\w_{r+1},\ldots,\w_s$ for the purpose of computing minimal generators.

\begin{lem}\label{winre}
Assume \autoref{reesdata}. Then after row operations on $\xi$ corresponding to an automorphism of $\k[\w_1,\ldots,\w_s]$ which fixes $\w_1,\ldots,\w_r$, we get $\ut^{\ua}\in\R(E)$ for all $\ua\in\N^s$ with $\ua^+=0$.
\end{lem}

\begin{proof}
By definition of $r$, for every $i>r$ we have $\sigma_i=0$. By \eqref{eq:res2}, this means the last $s-r$ rows of $\xi$ have entries in $\k$. Now by \autoref{fmode}(a,b), $\grade I_s(\xi)=\grade\Fitt_0(F/E)=\grade\ann(F/E)=2$, so $I_s(\xi)\neq0$. Thus there is an invertible $n\times n$ matrix $\chi$ with entries in $\k$ such that \[\xi\chi=\begin{bmatrix}A&B\\0&I_{s-r}\end{bmatrix}\] where $I_{s-r}$ is the $(s-r)\times(s-r)$ identity matrix, and $A$ and $B$ are arbitrary matrices of sizes $r\times(n-s+r)$ and $r\times(s-r)$, respectively. We may perform row operations on $\xi$ that involve subtracting multiples of the last $s-r$ rows from the first $r$ rows to assume that $B=0$. Such row operations correspond to automorphisms of $\k[\w_1,\ldots,\w_s]$ fixing $\w_1,\ldots,\w_r$. Now \autoref{fpoly} yields
\[\begin{bmatrix}T_1&\cdots&T_n\end{bmatrix}\chi=\begin{bmatrix}\w_1&\cdots&\w_s\end{bmatrix}\xi\chi=\begin{bmatrix}\w_1&\cdots&\w_s\end{bmatrix}\begin{bmatrix}A&0\\0&I_{s-r}\end{bmatrix},\] from which we get $T_{n-s+i}=\w_i$ for all $r<i\le s$. Because all $T_j\in\R(E)$, this means $\w_{r+1},\ldots,\w_s\in\R(E)$. Therefore if $\ua\in\N^s$ has $\ua^+=0$, meaning $\ua=(0,\ldots,0,\alpha_{r+1},\ldots,\alpha_s)$, we get $\ut^{\ua}=\w_{r+1}^{\alpha_{r+1}}\cdots\w_s^{\alpha_s}\in\R(E)$.
\end{proof}

Now we are ready to give minimal generating sets for $\R(M)$-modules of the form $\R(F)_{\ge-c,*}$.

\begin{dfn}\label{dfrfgen}
Assume \autoref{reesdata}, and fix $c\ge 0$. Set
\[\Lambda_c=\{\ua\in\N^s\mid\ua^0=0\text{ and }\as\ge c\}.\]
For $1\le i\le r$, let
\[\Omega_{c,i}=\{\ua\in\N^s\mid\alpha_i>0\text{ and }\alpha_j=0\text{ for }j>i\text{ and }c\le\as<c+\sigma_i\}\subset\Lambda_c\]
and set $\Omega_c=\bigcup_{i=1}^r\Omega_{c,i}\subset\Lambda_c$. Define
\[\mathcal{A}_c=\{\ut^{\ua}\mid\ua\in\N^s\text{ and }\ua^0=0\text{ and }\as<c\}\]
and
\[\mathcal{B}_c=\{x_0^jx_1^k\ut^{\ua}\mid\ua\in\Omega_c\text{ and }j+k=\as-c\}.\]
\end{dfn}

We will see in \autoref{rfgen} that $\Ac\cup\Bc$ is a minimal generating set for $\R(F)_{\ge-c,*}$. Observe that for any monomial $h\in\mathcal{A}_c$, we have $-c<\deg_xh\le0$, while any monomial $h\in\mathcal{B}_c$ has $\deg_xh=-c$. In fact, the set $\Ac$ consists of all monomials of $x$-degree $>-c$ which do not involve $x_0$, $x_1$, or any $\w_i$ with $i>r$ (which are the $i$ with $\sigma_i=0$). The set $\Bc$ is more complicated, but it may be thought of as the set of all monomials of $x$-degree $-c$ which do not involve any $\w_i$ with $i>r$, and only include just as many other $\w_i$ as necessary to make the $x$-degree small enough. This is made precise in the next lemma.

\begin{lem}\label{lambda0}
Assume \autoref{reesdata}, and fix $c>0$. Place a partial order on $\N^s$ such that $\ua\le\ub$ if and only if $\alpha_i\le\beta_i$ for all $i$. Then $\Omega_c$ is the set of minimal elements of $\Lambda_c$ under this partial order.
\end{lem}

\begin{proof}
First, we show that every minimal element of $\Lambda_c$ is in $\Omega_c$. Let $\ua$ be a minimal element of $\Lambda_c$. Choose $i$ so that $\alpha_i>0$ and $\alpha_j=0$ for all $j>i$. From the definition of $\Lambda_c$, we have $\ua^0=0$. This means $\alpha_{r+1}=\cdots=\alpha_s=0$, therefore $i\le r$. We claim that $\ua\in\Omega_{c,i}$. Suppose not. Then since $\as\ge c$ by definition of $\Lambda_c$, we must have $\as\ge c+\sigma_i$. Let $\ub=(\alpha_1,\ldots,\alpha_{i-1},\alpha_i-1,0,\ldots,0)=\ua-(0,\ldots,0,1,0,\ldots,0)$ (with the 1 in position $i$), which is in $\N^s$ since $\alpha_i>0$. Moreover, $\ub^0=0$, and $\bs=\as-\sigma_i\ge c$. Therefore $\ub\in\Lambda_c$. But $\ub<\ua$, contradicting the minimality of $\ua$. Hence $\ua\in\Omega_{c,i}\subset\Omega_c$.

Now we must show that every element of $\Omega_c$ is minimal in $\Lambda_c$. Take $\ua\in\Omega_c$; then $\ua\in\Omega_{c,i}$ for some $1\le i\le r$. Consider $\ub\in\Lambda_c$ with $\ub\le\ua$. Choose $k$ so that $\beta_k>0$ and $\beta_j=0$ for all $j>k$. Since $\ub\le\ua$, the largest nonzero index of $\ua$ must be greater than or equal to the largest nonzero index of $\ub$, that is, $i\ge k$. This means that for $j>i$, $\alpha_j=\beta_j=0$. Further, the inequalities $\bs\ge c$ and $\as<c+\sigma_i$ yield $(\ua-\ub)\cdot\us=\as-\bs<(c+\sigma_i)-c=\sigma_i$. Thus $\sigma_i>(\ua-\ub)\cdot\us=\sum_{j=1}^s(\alpha_j-\beta_j)\sigma_j=\sum_{j=1}^i(\alpha_j-\beta_j)\sigma_j$, since $\alpha_j=\beta_j$ for $j>i$. But each coefficient $\alpha_j-\beta_j$ is nonnegative, so if some $\alpha_\ell\neq\beta_\ell$, we would have $\sum_{j=1}^i(\alpha_j-\beta_j)\sigma_j\ge\sigma_\ell\ge\sigma_i$, the latter inequality because the $\sigma_j$ are in non-increasing order. This contradiction means that we must have $\alpha_j=\beta_j$ for all $j$, thus $\ua=\ub$. Hence $\ua$ is minimal in $\Lambda_c$.
\end{proof}

\begin{thm}\label{rfgen}
Assume \autoref{reesdata}, and fix $c>0$. Let $\mathcal{A}_c$ and $\mathcal{B}_c$ be as defined in \autoref{dfrfgen}.
\begin{itemize}
\item[(a)] The image of $\mathcal{A}_c$ is a minimal generating set for the $\R(M)$-module $\R(F)_{\ge-c,*}/R\R(F)_{-c,*}$.
\item[(b)] $\mathcal{B}_c$ is a minimal generating set for the $\R(M)_{0,*}$-module $\R(F)_{-c,*}$.
\item[(c)] $\mathcal{A}_c\cup\mathcal{B}_c$ is a minimal generating set for the $\R(M)$-module $\R(F)_{\ge-c,*}$.
\end{itemize}
\end{thm}

\begin{proof}
(a) Let $\overline{\Ac}$ be the image of $\Ac$ in $\R(F)_{\ge-c,*}/R\R(F)_{-c,*}$. Since $\R(F)_{\ge-c,*}$ is generated by monomials, to show that $\overline{\Ac}$ is a generating set, it suffices to show that for any monomial $h\in\R(F)$ with $\deg_xh\ge-c$, the image $\overline{h}\in\R(F)_{\ge-c,*}/R\R(F)_{-c,*}$ is in the $\R(M)$-module generated by $\overline{\Ac}$. Write $h=x_0^jx_1^k\ut^{\ub}$. First suppose $\ub^+\cdot\us<c$. Then since $(\ub^+)^0=0$, if we set $h'=\ut^{\ub^+}$ we have $h'\in\Ac$. Since $x_0^jx_1^k\ut^{\ub^0}$ has nonnegative $x$-degree, it is in $\R(M)$, therefore $\overline{h}=x_0^jx_1^k\ut^{\ub^0}\overline{h'}$ is in the $\R(M)$-module generated by $\overline{\Ac}$. On the other hand, suppose $c\le\ub^+\cdot\us=\bs$. Then since $-c\le\deg_xh=j+k-\bs$, we have $0\le\bs-c\le j+k$. Thus there are $0\le j'\le j$ and $0\le k'\le k$ such that $j'+k'=\bs-c$. If we set $h'=x_0^{j'}x_1^{k'}\ut^{\ub}$, then we have $\deg_xh'=j'+k'-\bs=-c$, meaning $h'\in\R(F)_{-c,*}$. Therefore $h=x_0^{j-j'}x_1^{k-k'}\in R\R(F)_{-c,*}$, so $\overline{h}=0$, which is certainly in the $\R(M)$-module generated by $\overline{\Ac}$. Hence $\overline{\Ac}$ generates $\R(F)_{\ge-c,*}/R\R(F)_{-c,*}$.

Because $\Ac$ consists of monomials, to see that it is a minimal generating set, it is enough to show that no element of $\Ac$ divides another. Suppose there are $h,h'\in\Ac$ with $h'\mid h$, meaning there is $\ell\in\R(M)$ with $h=h'\ell$. Write $h=\ut^{\ua}$ and $h'=\ut^{\ua'}$. The only way we can have $h=h'\ell$ is if $\ua'\le\ua$, in which case $\ell=\ut^{\ua-\ua'}$. Because $\ell\in\R(M)$, we must have $\deg\ell=-(\ua-\ua')\cdot\us\ge0$, so $(\ua-\ua')\cdot\us\le0$. But each term in the sum $(\ua-\ua')\cdot\us=\sum_{i=1}^s(\alpha_i-\alpha'_i)\sigma_i$ is nonnegative, so each $(\alpha_i-\alpha'_i)\sigma_i$ must be 0. If $i\le r$, then $\sigma_i>0$, so $\alpha_i=\alpha'_i$. If $i>r$, then $\alpha_i=\alpha'_i=0$ by the assumption that $\ua^0=(\ua')^0=0$. Thus $\ua=\ua'$ and therefore $h=h'$. Hence $\overline{\Ac}$ is a minimal generating set for $\R(F)_{\ge-c,*}/R\R(F)_{-c,*}$.

(b) Since every monomial in $\Bc$ has $x$-degree $-c$, we have $\Bc\subset\R(F)_{-c,*}$. Since $\R(F)_{-c,*}$ is generated by monomials, to show that $\Bc$ is a generating set it will suffice to show that any monomial $h\in\R(F)_{-c,*}$ is divisible by an element of $\Bc$. Write $h=x_0^jx_1^k\ut^{\ub}$. Then
\begin{equation}\label{eq:degc}-c=\deg_xh=j+k-\bs=j+k-\ub^+\cdot\us,\end{equation} so $\ub^+\cdot\us\ge c$. Also, $(\ub^+)^0=0$, therefore $\ub^+\in\Lambda_c$. By \autoref{lambda0}, there is $\ua\in\Omega_c$ such that $\ua\le\ub^+$. Since $\ua\in\Lambda_c$, $\as\ge c$. Thus $0\le\as-c\le\ub^+\cdot\us-c=j+k$ (using \eqref{eq:degc}), so there are $0\le j'\le j$ and $0\le k'\le k$ such that $j'+k'=\as-c$. Set $h'=x_0^{j'}x_1^{k'}\ut^{\ua}$. Then $h'\in\Bc$ since $\ua\in\Omega_c$ and $j'+k'=\as-c$. Because $j'\le j$, $k'\le k$, and $\ua\le\ub^+\le\ub$, there is $\ell\in\R(F)$ such that $h=\ell h'$. Since $\deg_x\ell=\deg_xh-\deg_xh'=(-c)-(-c)=0$, we have $\ell\in\R(M)_{0,*}$. This proves that $h$ is divisible by an element of $\Bc$, so $\Bc$ is a generating set for $\R(F)_{-c,*}$.

To show minimality, since $\Bc$ consists of monomials, it is enough to show that no element of $\Bc$ divides another. Suppose there are $h,h'\in\Bc$ with $h'\mid h$. Write $h=x_0^jx_1^k\ut^{\ua}$ and $h'=x_0^{j'}x_1^{k'}\ut^{\ua'}$. Then $\ua,\ua'\in\Omega_c$ and $\ua'\le\ua$. By \autoref{lambda0}, $\ua'=\ua$. Therefore $j'+k'=\ua'\cdot\us-c=\as-c=j+k$, but also $j'\le j$ and $k'\le k$, therefore $j'=j$ and $k'=k$. This shows that $h'=h$, hence $\Bc$ is a minimal generating set for $\R(F)_{-c,*}$.

(c) It follows from (b) that $\mathcal{B}_c$ minimally generates the $\R(M)$-module $L$ defined in (a), and $\mathcal{A}_c$ minimally generates $\R(F)_{\ge-c,*}/L$ by (a). Thus, taken together, $\Ac\cup\Bc$ generates $\R(F)_{\ge-c,*}$. For minimality, it is enough to show that no element of $\Ac\cup\Bc$ divides another. But no element of $\Ac$ is divisible by another element of $\Ac\cup\Bc$ by (a), and no element of $\Bc$ is divisible by another element of $\Bc$ by (b). Finally, no element of $\Bc$ may be divisible by an element of $\Ac$ since every element of $\Bc$ has $x$-degree $-c$ while every element of $\Ac$ has $x$-degree $>-c$. Thus $\Ac\cup\Bc$ is a minimal generating set for $\R(F)_{\ge-c,*}$.
\end{proof}

\subsection{Generators of the equations of the Rees algebra}

Now we will prove the main theorem. In \autoref{remint}, we showed the equality $\R(E)_{\ge d_m-1,*}=\R(M)_{\ge d_m-1,*}$. This allows us to compute the contraction from \autoref{kre} in $x$-degree $\ge d_m$, and thus determine the generators of $\K_{m+1}\R(E)$ in these degrees using \autoref{rfgen}.

It may seem like we should be able to compute the minimal generators in $x$-degree $d_m-1$ as well. However, because \autoref{rfgen} gives a generating set as an $\R(M)$-module, and we want a generating set as an $\R(E)$-module, the proof of \autoref{kregen} is not as simple as just applying \autoref{rfgen}. Instead, it works by factoring out the $x$-degree $d_m-1$ part of $\K_{m+1}\R(E)$, which only allows us to compute the generators in $x$-degree $\ge d_m$. In \S4, we will determine the generators in $x$-degree $d_m-1$ when $n=3$ and $m=1$.

\begin{thm}\label{kregen}
Assume \autoref{reesdata}. Then the set \begin{equation}\label{eq:kregen}\{g_{m+1}\ut^{\ua}\mid\ua\in\N^s\text{ and }\ua^0=0\text{ and }\as\le d_{m+1}-d_m\}\end{equation} is equal to the $x$-degree $\ge d_m$ part of a minimal generating set for $\K_{m+1}\R(E)$ as an $\R(E)$-module.
\end{thm}

\begin{proof}
Observe that the set \eqref{eq:kregen} is equal to $g_{m+1}\Ac$ for $c=d_{m+1}-d_m+1$, where $\Ac$ is defined in \autoref{dfrfgen}. Let $\J=\K_{m+1}\R(E)$. The claim that $g_{m+1}\mathcal{A}_c$ is the part of a minimal generating set of $\J$ having $x$-degree $\ge d_m$ is equivalent to saying that the image of $g_{m+1}\mathcal{A}_c$ in $\R(E)/R\J_{d_m-1,*}$ is a minimal generating set of $\J_{\ge d_m-1,*}/R\J_{d_m-1,*}$.

From \autoref{kre}, $\J=(g_{m+1}\R(F)_{\ge-d_{m+1},*})\cap\R(E)$. Therefore
\[\J_{\ge d_m-1,*}=(g_{m+1}\R(F))_{\ge d_m-1,*}\cap\R(E)_{\ge d_m-1,*}=(g_{m+1}\R(F)_{\ge-c,*})\cap\R(E)_{\ge d_m-1,*}.\]
But $\R(E)_{\ge d_m-1,*}=\R(F)_{\ge d_m-1,*}$ by \autoref{remint}, therefore
\begin{equation}\label{eq:jrf}\J_{\ge d_m-1,*}=g_{m+1}\R(F)_{\ge-c,*}.\end{equation}

We claim that it is enough to show that $\R(F)_{\ge-c,*}=\Ac\R(E)+R\R(F)_{-c,*}$. Indeed, if we multiply this equation by $g_{m+1}$ and use \eqref{eq:jrf}, we obtain $\J_{\ge d_m-1,*}=g_{m+1}\Ac\R(E)+R\J_{d_m-1,*}$. This shows that $g_{m+1}\Ac$ is a generating set for $\J_{\ge d_m-1,*}/R\J_{d_m-1,*}$ over $\R(E)$. Minimality follows because $\Ac$, and therefore $g_{m+1}\Ac$, is minimal over the larger ring $\R(M)$ (\autoref{rfgen}(a)).

We now prove that $\R(F)_{\ge-c,*}=\Ac\R(E)+R\R(F)_{-c,*}$. \autoref{rfgen}(a) says that $\Ac$ generates $\R(F)_{\ge-c,*}/R\R(F)_{c,*}$ as an $\R(M)$-module, that is to say, $\R(F)_{\ge-c,*}=\Ac\R(M)+R\R(F)_{-c,*}$. It is clear that $\Ac\R(E)+R\R(F)_{-c,*}\subset\Ac\R(M)+R\R(F)_{-c,*}$, so we just need to show that $\Ac\R(M)\subset\Ac\R(E)+R\R(F)_{-c,*}$. Because $\R(M)$ is generated by monomials, it suffices to show that for any monomials $h\in\Ac$ and $f\in\R(M)$, we have $hf\in\Ac\R(E)+R\R(F)_{-c,*}$.

By \autoref{winre}, we may apply an automorphism of $\k[\w_1,\ldots,\w_s]$ which fixes $\w_1,\ldots,\w_r$, and therefore fixes $\Ac$, to assume that $\ut^{\ub^0}\in\R(E)$ for all $\ub$. From the definition of $\Ac$, we may write $h=\ut^{\ua}$ with $\ua^0=0$ and $\as<c$. Also, write $f=x_0^jx_1^k\ut^{\ub}$, so that $hf=x_0^jx_1^k\ut^{\ua+\ub}$. If $(\ua+\ub)\cdot\us<c$, then since $(\ua+\ub)\cdot\us=(\ua+\ub^+)\cdot\us$, we have $\ut^{\ua+\ub^+}\in\Ac$. Thus $hf$ is the product of $\ut^{\ua+\ub^+}\in\Ac$ and $x_0^jx_1^k\ut^{\ub^0}\in\R(E)$, meaning $hf\in\Ac\R(E)$.

On the other hand, if $(\ua+\ub)\cdot\us\ge c$, then $d:=\deg_x(\ut^{\ua+\ub})\le-c$. Since $\deg_xh>-c$ (as $h\in\Ac$) and $\deg_xf\ge0$ (as $f\in\R(M)$), we obtain $\deg_x(hf)\ge-c$. But $\deg_x(hf)=\deg_x(x_0^jx_1^k\ut^{\ua+\ub})=j+k+d$, so $0\le-c-d\le j+k$. Thus there are $0\le j'\le j$ and $0\le k'\le k$ such that $j'+k'=-c-d$. Set $\ell=x_0^{j'}x_1^{k'}\ut^{\ua+\ub}$. Then $\deg_x\ell=j'+k'+d=-c$, so $\ell\in\R(F)_{-c,*}$. Therefore $hf=(x_0^{j-j'}x_1^{k-k'})\ell\in R\R(F)_{-c,*}$.

In either case, $hf\in\Ac\R(E)+R\R(F)_{-c,*}$, which proves the claim that $\R(F)_{\ge-c,*}=\Ac\R(E)+R\R(F)_{-c,*}$. Therefore the image of $g_{m+1}\Ac$ minimally generates $\J_{\ge d_m,*}/R\J_{d_m-1,*}$ as an $\R(E)$-module. For every $\ua\in\N^s$ with $\ua^0=0$ and $\as<c$, we get a generator $g_{m+1}\ut^{\ua}\in g_{m+1}\mathcal{A}_c$ having bidegree $(d_{m+1}-\as,\sum_{i=1}^r\alpha_i+1)$.
\end{proof}

The bidegrees of the minimal generators given in \autoref{kregen} may be computed:

\begin{cor}\label{bideg}
Assume \autoref{reesdata}. Then $\K_{m+1}\R(E)$ has a minimal generating set whose elements of $x$-degree $\ge d_m$ lie in bidegrees $(d_{m+1}-\as,\sum_{i=1}^r\alpha_i+1)$ for each $\ua\in\N^s$ with $\ua^0=0$ and $\as\le d_{m+1}-d_m$. In particular, the only minimal generator of $\K_{m+1}\R(E)$ having $x$-degree $d_{m+1}$ is $g_{m+1}$, and all other minimal generators of $\K_{m+1}\R(E)$ have $x$-degree $<d_{m+1}$.
\end{cor}

\begin{proof}
Since $g_{m+1}$ has bidegree $(d_{m+1},1)$, for every $\ua\in\N^s$ with $\ua^0=0$ and $\as\le d_{m+1}-d_m$, the element $g_{m+1}\ut^{\ua}$ has bidegree $(d_{m+1}-\as,\sum_{i=1}^r\alpha_i+1)$, so the first claim follows from \autoref{kregen}. The second claim comes from noticing that if $\ua^0=0$ ad $\ua\neq\underline{0}$, then $\deg_x(\ut^{\ua})<0$, therefore $\deg_x(g_{m+1}\ut^{\ua})<d_{m+1}$.
\end{proof}

As a consequence, we get \autoref{kregencor}, which gives the bidegrees of the minimal generators of $\K$ having $x$-degree at least $d_{n-2}$. For example pictures of these bidegrees, see Tables 1, 2, and 3. The row represents $T$-degree, the column represents $x$-degree, and the numbers in the table are the number of generators in the corresponding bidegree. Only the minimal generators lying on the right of the vertical line are given in \autoref{kregencor}. Note the patterns in these tables: in cases (i) and (iii), the generators lie on a line of slope $-1/\sigma_1$; while in case (ii), the generators lie between a line of slope $-1/\sigma_1$ and a line of slope $-1/\sigma_2$.

\begin{cor}\label{kregencor}
Assume \autoref{reesdata} with $m=n-2$. Then $s=2$ and $\sigma_1+\sigma_2=\sum_{i=1}^{n-2}d_i$, and the minimal generators of $\K$ having $x$-degree $\ge d_{n-2}$ consist of the generator $g_{n-2}$ and any other $g_m$ having bidegree $(d_{n-2},1)$, and also:
\begin{itemize}
\item[(i)] If $\sigma_2=0$: one generator of bidegree $(d_{n-1}-j\sigma_1,j+1)$ for each integer $j$ with $0\le j\le\frac{d_{n-1}-d_{n-2}}{\sigma_1}$.
\item[(ii)] If $\sigma_1>\sigma_2>0$: one generator of bidegree $(d_{n-1}-i\sigma_1-(j-i)\sigma_2,j+1)$ for each pair of integers $(j,i)$ with $0\le j\le\frac{d_{n-1}-d_{n-2}}{\sigma_2}$ and $0\le i\le\min\{j,\frac{d_{n-1}-d_{n-2}-j\sigma_2}{\sigma_1-\sigma_2}\}$.
\item[(iii)] If $\sigma_1=\sigma_2$: $j+1$ generators of bidegree $(d_{n-1}-j\sigma_1,j+1)$ for each integer $j$ with $0\le j\le\frac{d_{n-1}-d_{n-2}}{\sigma_1}$.
\end{itemize}
\end{cor}

\begin{table}
\begin{tabular}{c|ccc|cccccccccccccc}
7&&&&\\
6&&&&\\
5&&&&&1\\
4&&&&&&&&1\\
3&&&&&&&&&&&1\\
2&&&&&&&&&&&&&&1\\
1&&&&1&&&&&&&&&&&&&1\\
\hline
&0&1&2&3&4&5&6&7&8&9&10&11&12&13&14&15&16
\end{tabular}
\caption{An example of case (i), with $n=3$, $d_1=3$, $d_2=16$, $\sigma_1=3$, and $\sigma_2=0$.}

\begin{tabular}{c|ccccc|cccccccccccc}
7&&&&&&\\
6&&&&&&1&1\\
5&&&&&&1&1&1&1\\
4&&&&&&&&1&1&1&1\\
3&&&&&&&&&&&1&1&1\\
2&&&&&&&&&&&&&&1&1\\
1&&&&&&1&&&&&&&&&&&1\\
\hline
&0&1&2&3&4&5&6&7&8&9&10&11&12&13&14&15&16
\end{tabular}
\caption{An example of case (ii), with $n=3$, $d_1=5$, $d_2=16$, $\sigma_1=3$, and $\sigma_2=2$.}

\begin{tabular}{c|cccc|ccccccccccccc}
7&&&&&7\\
6&&&&&&&6\\
5&&&&&&&&&5\\
4&&&&&&&&&&&4\\
3&&&&&&&&&&&&&3\\
2&&&&&&&&&&&&&&&2\\
1&&&&&1&&&&&&&&&&&&1\\
\hline
&0&1&2&3&4&5&6&7&8&9&10&11&12&13&14&15&16
\end{tabular}
\caption{An example of case (iii), with $n=3$, $d_1=4$, $d_2=16$, $\sigma_1=2$, and $\sigma_2=2$.}
\end{table}

\begin{proof}
The equalities $s=2$ and $\sigma_1+\sigma_2=\sum_{i=1}^{n-2}d_i$ come from \autoref{fmode}. Since $\K=\K_{n-1}$, a generating set for $\K$ consists of the lifts of generating sets of each quotient $\K_{m+1}/\K_m\cong\K_{m+1}\R(E_m)$ for $1\le m\le n-2$. For $m<n-2$, \autoref{bideg} says that besides $g_{m+1}$, all the minimal generators of $\K_{m+1}\R(E_m)$ have $x$-degree less than $d_{m+1}$. Since $d_{m+1}\le d_{n-2}$, the only generator of $\K_{m+1}\R(E_m)$ that could possibly have $x$-degree at least $d_{n-2}$ is $g_{m+1}$. The remaining generators with $x$-degree at least $d_{n-2}$ must come from generators of $\K_{n-1}\R(E_{n-2})$.

The bidegrees of the minimal generators of $\K_{n-1}\R(E_{n-2})$ are given in \autoref{bideg}. If $\sigma_2=0$, then $r=1$, so we only consider those pairs $(\alpha_1,\alpha_2)$ with $\alpha_2=0$. Setting $j=\alpha_1$, we see that $\K_{n-1}\R(E_{n-2})$ has a minimal generator in bidegree $(d_{n-1}-j\sigma_1,j+1)$ for all $j$ with $j\sigma_1\le d_{n-1}-d_{n-2}$. If $\sigma_2>0$, then $r=2$. Set $i=\alpha_1$ and $j=\alpha_1+\alpha_2$. Then $\K_{n-1}\R(E_{n-2})$ has a minimal generator in bidegree $(d_{n-1}-i\sigma_1-(j-i)\sigma_2,j+1)$ for all $i,j$ with $i\sigma_1+(j-i)\sigma_2\le d_{n-1}-d_{n-2}$. If $\sigma_1>\sigma_2$, then these bidegrees are all distinct. For if $(d_{n-1}-i\sigma_1-(j-i)\sigma_2,j+1)=(d_{n-1}-i'\sigma_1-(j'-i')\sigma_2,j'+1)$, then certainly $j=j'$. But then $i(\sigma_1-\sigma_2)=i'(\sigma_1-\sigma_2)$, whence $i=i'$. If $\sigma_1=\sigma_2$, on the other hand, these bidegrees reduce to $(d_{n-1}-j\sigma_1,j+1)$, and there is one generator in this bidegree for each $0\le i\le j$, for a total of $j+1$ generators.
\end{proof}

The condition $\sigma_2=0$ in \autoref{kregencor} has geometric significance, as we see in the next proposition, a consequence of the results in \cite{ckpu13}. The proposition only holds if $d_{n-2}<d_{n-1}$, for if $d_{n-2}=d_{n-1}$, then the matrix $\varphi_{n-2}$ consisting of all but the last column of $\varphi$ is not an invariant. However, if $d_{n-2}=d_{n-1}$, then the only minimal generators of $\K$ having $x$-degree $\ge d_{n-2}$ are some of the $g_m$, so there is no difference between the cases in \autoref{kregencor}. The birationality assumption is also mild, as any rational curve $\C$ may be reparametrized so that the morphism $\Phi:\P^1\to\C$ is birational, as shown in \cite{kpu14}. It is not known whether there is any geometric meaning of the values of $\sigma_1$ and $\sigma_2$ when they are both positive.

\begin{prop}
Assume \autoref{reesdata} with $m=n-2$. Let $f_1,\ldots,f_n$ be a minimal generating set of $I$ and let $\Phi:\P^1\to \C\subset\P^{n-1}$ be the morphism defined by $(f_1:\cdots:f_n)$. Assume $\k$ is algebraically closed and $\Phi$ is birational onto its image $\C$, and assume $d_{n-2}<d_{n-1}$. Then $\sigma_2=0$ if and only if $\C$ has a point of multiplicity greater than $d_{n-2}$ (which must then have multiplicity $d_{n-1}$).
\end{prop}

\begin{proof}
Let $p=(a_1:\cdots:a_n)\in\P^{n-1}$ be a point. By \cite[Theorem~1.8]{ckpu13}, the multiplicity of $p$ on $\C$ is $\deg\gcd I_1(\underline{a}\varphi)$ where $\underline{a}=[a_1\ \cdots\ a_n]$. We may write $\underline{a}\varphi=[h_1\ \cdots\ h_{n-1}]$ where $h_i$ is homogeneous of degree $d_i$, so that the multiplicity of $p$ on $\C$ is $\deg\gcd(h_1,\ldots,h_{n-1})$. By assumption, $d_i<d_{n-1}$ for $1\le i\le n-2$. Thus the only way that the gcd of $h_1,\ldots,h_{n-1}$ can have degree greater than $d_{n-2}$ is if $h_1=\cdots=h_{n-2}=0$, in which case the gcd will have degree $d_{n-1}$. Therefore $\C$ has a point of multiplicity greater than $d_{n-2}$ if and only if some linear combination of the rows of $\varphi$ is zero except for the last column. This is the same as saying that some linear combination of the rows of $\varphi_{n-2}$ is zero. From the exact sequence \eqref{eq:res2}, this will occur if and only if $\sigma_2=0$.
\end{proof}

\subsection{A recursive procedure for the generators}

While \autoref{kregen} gives the generators as elements of $R[\w_1,\ldots,\w_s]$, we seek the generators as elements of $S=R[T_1,\ldots,T_n]$. In \autoref{kregenalg}, we show how the generators in $S$ may be computed. For specific examples, see \autoref{exgen0} and \autoref{ex2n}. The algorithm hinges on the following way to multiply by $w_i$ in $S$.

\begin{prop}\label{wmult}
Assume \autoref{reesdata}. For $1\le i\le s$, let $\xi_i$ be the $i^{th}$ row of $\xi$ and let $\xi^i$ be the matrix obtained by deleting the $i^{th}$ row from $\xi$. Let $\rho^i=[\rho^i_1\ \cdots\ \rho^i_{m+1}]$ be the $n\times(m+1)$ matrix whose columns generate the kernel of $\xi^i$. Set $p^i_j=\xi_i\rho^i_j\in R$ and set $q^i_j=[T_1\ \cdots\ T_n]\rho^i_j\in S_{*,1}$. Then
\begin{enumerate}
\item[(a)] For each $i$, $(p^i_1,\ldots,p^i_{m+1})_{d_m-1+\sigma_i}=R_{d_m-1+\sigma_i}$.
\item[(b)] For each $i$ and $j$, we have $q^i_j=p^i_jw_i$ in $\R(E)\subset\R(F)$.
\end{enumerate}
\end{prop}

\begin{proof}
(a) Let $h\in R$ be a form of degree $d_m-1+\sigma_i$. Then $[0\ \cdots\ h\ \cdots\ 0]\tr\in F$ (with $h$ in the $i^{\mathrm{th}}$ position) has degree $d_m-1$. By \autoref{fmode}, $(F/E)_{d_m-1}=0$, and since $F/E=\coker\xi$ by \eqref{eq:res2}, $\xi$ is surjective in degree $d_m-1$. Thus there is $u\in R^n$ such that
\begin{equation}\label{eq:xiu}\xi u=\begin{bmatrix}0\\\vdots\\h\\\vdots\\0\end{bmatrix}.\end{equation}
This means that $u\in\ker\xi^i=\im\rho^i$, so $u=a_1\rho^i_1+\cdots+a_{m+1}\rho^i_{m+1}$ for some $a_1,\ldots,a_{m+1}\in R$. Then the $i^{\mathrm{th}}$ row of \eqref{eq:xiu} yields
\[h=\xi_i u=a_1\xi_i\rho^i_1+\cdots+a_{m+1}\xi_i\rho^i_{m+1}=a_1p^i_1+\cdots+a_{m+1}p^i_{m+1}.\]
Therefore $h\in(p^i_1,\ldots,p^i_{m+1})$ for all $h\in R_{d_m-1+\sigma_i}$, hence $(p^i_1,\ldots,p^i_{m+1})_{d_m-1+\sigma_i}=R_{d_m-1+\sigma_i}$.

(b) Since $\rho^i_j$ is in the kernel of $\xi^i$ (the matrix consisting of all but the $i^{\mathrm{th}}$ row of $\xi$),
\[\xi\rho^i_j=\begin{bmatrix}0\\\vdots\\\xi_i\rho^i_j\\\vdots\\0\end{bmatrix}=\begin{bmatrix}0\\\vdots\\p^i_j\\\vdots\\0\end{bmatrix}\] with $p^i_j$ appearing in the $i$th row. Thus, using \autoref{fpoly},
\[q^i_j=\begin{bmatrix}T_1&\cdots&T_n\end{bmatrix}\rho^i_j=\begin{bmatrix}\w_1&\cdots&\w_s\end{bmatrix}\xi\rho^i_j=\begin{bmatrix}\w_1&\cdots&\w_s\end{bmatrix}\begin{bmatrix}0\\\vdots\\p^i_j\\\vdots\\0\end{bmatrix}=p^i_j\w_i.\qedhere\]
\end{proof}

\begin{cor}\label{kregenalg}
Assume \autoref{reesdata}. Define the polynomials $p^i_j$ and $q^i_j$ as in \autoref{wmult}. For each $\ua\in\N^s$ with $\ua^0=0$ and $\as\le d_{m+1}-d_m$, we recursively construct forms $h_{\ua}\in S$ of bidegree $(d_{m+1}-\as,\sum_{i=1}^r\alpha_i+1)$ as follows. Set $h_{\underline{0}}=g_{m+1}$. For $\ua\neq\underline{0}$, choose any $i$ such that $\alpha_i>0$, and assume we have constructed $h_{(\alpha_1,\ldots,\alpha_i-1,\ldots,\alpha_s)}$. Write $h_{(\alpha_1,\ldots,\alpha_i-1,\ldots,\alpha_s)}=a_1p^i_1+\cdots+a_{m+1}p^i_{m+1}$ for some $a_1,\ldots,a_{m+1}\in S$ (which we may do by \autoref{wmult}(a) since the $x$-degree of $h_{(\alpha_1,\ldots,\alpha_i-1,\ldots,\alpha_s)}$ is at least $d_m+\sigma_i$), then set $h_{\ua}=a_1q^i_1+\cdots+a_{m+1}q^i_{m+1}$. Then the images of $\{h_{\ua}\}$ in $\R(E)$ are uniquely determined and form the $x$-degree $\ge d_m$ part of a minimal generating set of $\K_{m+1}\R(E)$.
\end{cor}

\begin{proof}
Let $A=\{\ua\in\N^s\mid\ua^0=0\text{ and }\as\le d_{m+1}-d_m\}$. By \autoref{kregen}, the $x$-degree $\ge d_m$ part of a minimal generating set for $\K_{m+1}\R(E)$ is given by $\{g_{m+1}\ut^{\ua}\mid\ua\in A\}$. Thus we are done if we can show that $h_{\ua}=g_{m+1}\ut^{\ua}$ in $\R(E)$. We use induction on $\ell:=\sum\alpha_i$. If $\ell=0$, then $\ua=\underline{0}$, and by definition $h_{\underline{0}}=g_{m+1}=g_{m+1}\ut^{\underline{0}}$.

Now suppose $\ell>0$. By construction, there is $i$ for which $\alpha_i>0$ and $a_1,\ldots,a_{m+1}\in S$ such that $h_{(\alpha_1,\ldots,\alpha_i-1,\ldots,\alpha_s)}=a_1p^i_1+\cdots+a_{m+1}p^i_{m+1}$ and $h_{\ua}=a_1q^i_1+\cdots+a_{m+1}q^i_{m+1}$. Then
\begin{align*}
h_{\ua}&=a_1q^i_1+\cdots+a_{m+1}q^i_{m+1}\\
&=a_1p^i_1w_i+\cdots+a_{m+1}p^i_{m+1}w_i\text{ by \autoref{wmult}(b)}\\
&=h_{(\alpha_1,\ldots,\alpha_i-1,\ldots,\alpha_s)}\w_i\\
&=(g_{m+1}\w_1^{\alpha_1}\cdots\w_i^{\alpha_i-1}\cdots \w_s^{\alpha_s})w_i\text{ by induction}\\
&=g_{m+1}\ut^{\ua}.
\end{align*}
Therefore $h_{\ua}=g_{m+1}\ut^{\ua}$ for all $\alpha\in A$.
\end{proof}

When $m=1$, the process in \autoref{kregenalg} may be represented in terms of \emph{Sylvester forms} (or \emph{Jacobian duals}). In fact, when $n=3$ (so that there are only two $\sigma$s) and $\sigma_2=0$, the procedure in \autoref{kregenalg} agrees with the one given in \cite[Theorem~2.10]{cd13b}. It is also similar to the iterated Jacobian dual construction in \cite[\S4]{bm15}, although we use Sylvester forms with respect to multiple regular sequences $p^i_1,p^i_2$. Before stating the result, we give the definition of Sylvester forms.

Let $p_1,p_2\in R$ be homogeneous polynomials which form a regular sequence. Consider forms $f,g\in S$ such that $f,g\in(p_1,p_2)S$. Then we may (non-uniquely) write $f=f_1p_1+f_2p_2$ and $g=g_1p_1+g_2p_2$. We will, by abuse of notation, refer to the determinant
\[\syl_{p_1,p_2}(f,g)=\det\begin{bmatrix}f_1&g_1\\f_2&g_2\end{bmatrix}\]
as the {\em Sylvester form} of $f,g$ with respect to $p_1,p_2$. The Sylvester form $\syl_{p_1,p_2}(f,g)$ is not uniquely determined, but its image in $S/(f,g)$ is unique (\cite[Proposition~3.8.1.6]{jou95}).

\begin{prop}
Assume \autoref{reesdata} with $m=1$. Define $p^i_1,p^i_2$ as in \autoref{wmult}. Then for all $i$, the polynomials $p^i_1,p^i_2$ form a regular sequence, and for any form $h\in S$ with $x$-degree $\ge d_1-1+\sigma_i$, the image of the Sylvester form $\syl_{p^i_1,p^i_2}(g_1,h)$ in $\R(E)=S/(g_1)$ is a nonzero scalar multiple of $hw_i$. In particular, the polynomials $h_{\ua}$ of \autoref{kregenalg} may be defined, up to scalar multiples, as $h_{\ua}=\syl_{p^i_1,p^i_2}(g_1,h_{(\alpha_1,\ldots,\alpha_i-1,\ldots,\alpha_s)})$ for any $i$ such that $\alpha_i>0$.
\end{prop}

\begin{proof}
That $p^i_1,p^i_2$ form a regular sequence is immediate from \autoref{wmult}(a). Recall the matrix $\rho^i$ from \autoref{wmult}. Let $\delta^i_1,\delta^i_2$ be the degrees of the columns of $\rho^i$. We claim that $\delta^i_1+\delta^i_2=d_1-\sigma_i$. First, since $\xi^i$ is all but the $i$th row of $\xi$, there is a surjection $\coker\xi\twoheadrightarrow\coker\xi^i$. By \autoref{fmode}(a,b), $\coker\xi$ is Artinian, so $\coker\xi^i$ is Artinian as well. Since $R$ is a two-dimensional regular ring, there is a free resolution \[0\to R(-\delta^i_1)\oplus R(-\delta^i_2)\xrightarrow{\rho^i}R^n\xrightarrow{\xi^i}\bigoplus_{j\neq i}R(\sigma_j)\to\coker\xi^i\to0.\] By computing the Hilbert polynomial and using that $\coker\xi^i$ is Artinian, we see that $\delta^i_1+\delta^i_2=\sum_{j\neq i}\sigma_j$. But $\sum\sigma_j=d_1$ by \autoref{fmode}(e), hence $\delta^i_1+\delta^i_2=d_1-\sigma_i$.

Now let $h\in\R(E)$ be homogeneous of $x$-degree $\ge d_1-1+\sigma_i$. Then by \autoref{wmult}(a), there are $a_1,a_2\in\R(E)$ such that $h=a_1p^i_1+a_2p^i_2$. Recall that $\rho^i$ generates the kernel of $\xi^i$, the matrix obtained by removing the $i$th row of $\xi$. Since $\xi\varphi_1=0$ from \eqref{eq:res2}, $\varphi_1$ is in the kernel of $\xi^i$, therefore $\varphi_1=\rho^i\lambda$ for some $2\times1$ column vector $\lambda$ with entries in $R$. Then
\[\begin{bmatrix}p^i_1&p^i_2\end{bmatrix}\lambda=\xi_i\rho^i\lambda=\xi_i\varphi_1=0.\]
Since $p^i_1,p^i_2$ are a regular sequence, this means \[\lambda=\begin{bmatrix}-cp^i_2\\cp^i_1\end{bmatrix}\]
for some $c\in R$. Thus \begin{equation}\label{eq:g1fac}g_1=\begin{bmatrix}T_1&\cdots&T_n\end{bmatrix}\varphi_1=\begin{bmatrix}T_1&\cdots&T_n\end{bmatrix}\rho^i\lambda=\begin{bmatrix}q^i_1&q^i_2\end{bmatrix}\lambda=c(p^i_1q^i_2-p^i_2q^i_1).\end{equation}
Now $p^i_j$ has $x$-degree $\sigma_i+\delta_j$, while $q^i_j$ has $x$-degree $\delta_j$. Because we have shown that $\delta^i_1+\delta^i_2=d_1-\sigma_i$, we have $\deg_x(p^i_1q^i_2-p^i_2q^i_1)=\sigma_i+\delta^i_1+\delta^i_2=d_1=\deg_xg_1$, hence $c\in\k$. We cannot have $c=0$, or we would get $\varphi_1=\rho^i\lambda=\rho^i\cdot0=0$, which is impossible. Therefore $c$ is a unit. Since $h=a_1p^i_1+a_2p^i_1$ and $g_1=(cq^i_2)p^i_1+(-cq^i_1)p^i_2$ by \eqref{eq:g1fac}, the definition of the Sylvester form yields
\[\syl_{p^i_1,p^i_2}(g_1,h)=\det\begin{bmatrix}cq^i_2&a_1\\-cq^i_1&a_2\end{bmatrix}=c(a_1q^i_1+a_2q^i_2).\]
But by \autoref{wmult}(b), $a_1q^i_1+a_2q^i_2=a_1p^i_1w_i+a_2p^i_2w_i=hw_i$, so $\syl_{p^i_1,p^i_2}(g_1,h)=chw_i$. Because we showed in the proof of \autoref{kregenalg} that $h_{\ua}=h_{(\alpha_1,\ldots,\alpha_i-1,\ldots,\alpha_s)}w_i$ in $\R(E)$, the final claim follows.
\end{proof}

\subsection{Some examples}

\begin{ex}[compare {\cite[Theorem~3.6]{kpu11}}]\label{exalp}
Assume $I$ is {\em almost linearly presented}, meaning $d_1=\cdots=d_{n-2}=1$ and $d_{n-1}=c\ge1$. First, \autoref{fmode}(e) gives $\sigma_1+\sigma_2=\sum_{i=1}^{n-2}d_i=n-2$. By \autoref{remint}, the modules $E=E_{n-2}$ and $M=\m^{\sigma_1}(\sigma_1)\oplus\m^{\sigma_2}(\sigma_2)$ agree in degree $\ge0$, but they are both generated in degree 0, so $E\cong M$. Thus by \autoref{rns}, after an automorphism of $S=R[T_1,\ldots,T_n]$, the ideal of equations of $\R(E)$ is generated by the $2\times2$ minors of the matrix
\[\left[\begin{array}{c|cccc|cccc}x_0&T_1&T_2&\cdots&T_{\sigma_1}&T_{\sigma_1+2}&T_{\sigma_1+3}&\cdots&T_{n-1}\\x_1&T_2&T_3&\cdots&T_{\sigma_1+1}&T_{\sigma_1+3}&T_{\sigma_1+4}&\cdots&T_n\end{array}\right].\]
In particular, there are $n-2$ generators of bidegree $(1,1)$ and $\frac12(n-2)(n-3)$ generators of bidegree $(0,2)$. The ideal $\K$ of equations of $\R(I)$ is generated by these equations, plus a lift of a minimal generating set of $\K\R(E)$.

By \eqref{eq:jrf} with $m=n-2$, this ideal is $\K\R(E)=g_{n-1}\R(F)_{\ge-c,*}$ where $F=R(\sigma_1)\oplus R(\sigma_2)$. We may compute the generators of this ideal directly from \autoref{rfgen}: they are just the elements of $g_{n-1}\Ac\cup g_{n-1}\Bc$. The set $g_{n-1}\Ac$ consists of the minimal generators of $\K\R(E)$ having $x$-degree $>0$, and equals
\[\{g_{n-1}w_1^{\alpha_1}w_2^{\alpha_2}\mid\alpha_1\sigma_1+\alpha_2\sigma_2\le c-1\}\text{ if }\sigma_2>0\]
\[\{g_{n-1}w_1^{\alpha_1}\mid\alpha_1\sigma_1\le c-1\}\text{ if }\sigma_2=0.\]
We can translate from $w$'s to $T$'s by the equations $T_{j+1}=x_0^{\sigma_1-j}x_1^jw_1$ and $T_{\sigma_1+j+2}=x_0^{\sigma_2-j}x_1^jw_2$ (see \autoref{rns}). The bidegrees of these minimal generators are listed in \autoref{kregencor}.

We can also compute the minimal generators of $x$-degree 0, which are the elements of $g_{n-1}\Bc$. To do this, we must first determine the set $\Omega_c$ of \autoref{dfrfgen}. There are two cases. First, suppose $\sigma_2=0$. Then $r=1$, so
\[\Omega_c=\Omega_{c,1}=\{(\beta_1,0)\mid c\le\beta_1\sigma_1<c+\sigma_1\}.\]
This is just $\Omega_c=\{(k,0)\}$ where $k=\lceil\frac{c}{\sigma_1}\rceil$. Let $\ell=k\sigma_1-c$; note that $0\le\ell<\sigma_1$. Then $g_{n-1}\Bc$ consists of $\ell+1$ elements of bidegree $(0,k+1)$, namely
\[\{g_{n-1}x_0^{\ell-j}x_1^jw_1^k\mid0\le j\le\ell\}.\]

Second, suppose $\sigma_2>0$. Then $r=2$, so $\Omega_c=\Omega_{c,1}\cup\Omega_{c,2}$. As above, $\Omega_{c,1}=\{(k,0)\}$ where $k=\lceil\frac{c}{\sigma_1}\rceil$. Again, setting $\ell=k\sigma_1-c$, $g_{n-1}\Bc$ has $\ell+1$ elements of bidegree $(0,k+1)$, namely
\[\{g_{n-1}x_0^{\ell-j}x_1^jw_1^k\mid0\le j\le\ell\}.\]
In this case, however, we also have
\[\Omega_{c,2}=\{(\alpha_1,\alpha_2)\mid\alpha_2>0\text{ and }c\le\alpha_1\sigma_1+\alpha_2\sigma_2<c+\sigma_2\}.\]
To have $(\alpha_1,\alpha_2)\in\Omega_{c,2}$, we must have $\alpha_1\sigma_1<c+\sigma_2$, thus $\alpha_1<\lceil\frac{c+\sigma_2}{\sigma_1}\rceil$. Once we have chosen $\alpha_1$, the condition $c-\alpha_1\sigma_1\le\alpha_2\sigma_2<c-\alpha_1\sigma_1+\sigma_2$ forces $\alpha_2=\lceil\frac{c-\alpha_1\sigma_1}{\sigma_2}\rceil$. Hence $(\alpha_1,\alpha_2)\in\Omega_{c,2}$ come in pairs of the form $(i,v(i))$ where $0\le i<\lceil\frac{c+\sigma_2}{\sigma_1}\rceil$ and $v(i)=\lceil\frac{c-i\sigma_1}{\sigma_2}\rceil$. Let $\ell(i)=i\sigma_1+v(i)\sigma_2-c$, so that $0\le\ell(i)<\sigma_2$. Therefore the remaining part of $g_{n-1}\mathcal{B}_c$ consists of $\ell(i)+1$ generators of bidegree $(0,i+v(i)+1)$ for each $0\le i<\lceil\frac{c+\sigma_2}{\sigma_1}\rceil$, namely \[\{g_{n-1}x_0^{\ell(i)-j}x_1^jw_1^iw_2^{v(i)}\mid0\le i<\left\lceil\frac{c+\sigma_2}{\sigma_1}\right\rceil\}.\qedhere\]
\end{ex}

\begin{ex}[compare {\cite[Theorem~2.10]{cd13b}} and {\cite[Corollary~3.10]{kpu13}}]\label{exgen0}
Assume $n=3$ and assume that the entries of the first column of $\varphi$ are linearly dependent. Thus after row operations on $\varphi$, we have
\[\varphi=\begin{bmatrix}\gamma_1&\delta_1\\\gamma_2&\delta_2\\0&\delta_3\end{bmatrix}.\]
Since $\varphi_1=[\gamma_1\ \gamma_2\ 0]\tr$, the resolution \eqref{eq:res1} becomes
\[0\to R(-d_1)\oplus R\xrightarrow{\xi\tr}R^3\xrightarrow{\varphi_1\tr}R(d_1)\to\Ext^2_R(F/E,R)\to0\]
where the image of $\xi\tr$ is the kernel of $[\gamma_1\ \gamma_2\ 0]$, meaning
\[\xi=\begin{bmatrix}-\gamma_2&\gamma_1&0\\0&0&1\end{bmatrix}.\]
In particular, $\sigma_1=d_1$, $\sigma_2=0$, $s=2$, and $r=1$. Then by \autoref{kregen}, the minimal generators of $\K$ besides $g_1$ having $x$-degree $\ge d_1$ are $g_2\w_1^{\alpha_1}$ for $\alpha_1\in\N$ such that $\alpha_1\sigma_1\le d_2-d_1$, or in other words, $0\le\alpha_1\le\frac{d_2}{d_1}-1$.

To compute these generators, we will use \autoref{kregenalg}. The matrix $\rho^1$ defined in \autoref{wmult} is the syzygy matrix of the matrix obtained by deleting the first row of $\xi$, which is $[0\ 0\ 1]$, thus
\[\rho^1=\begin{bmatrix}1&0\\0&1\\0&0\end{bmatrix}.\]
Following \autoref{wmult}, we have
\[\begin{bmatrix}p^1_1&p^1_2\end{bmatrix}=\begin{bmatrix}-\gamma_2&\gamma_1&0\end{bmatrix}\rho^1=\begin{bmatrix}-\gamma_2&\gamma_1\end{bmatrix}\]
and
\[\begin{bmatrix}q^1_1&q^1_2\end{bmatrix}=\begin{bmatrix}T_1&T_2&T_3\end{bmatrix}\rho^1=\begin{bmatrix}T_1&T_2\end{bmatrix}.\]
Now \autoref{kregenalg} says that the generators of $\K$ can be computed recursively. Let $h_{(0,0)}=g_2$, and for any $1\le\alpha_1\le\frac{d_2}{d_1}-1$, write $h_{(\alpha_1-1,0)}=a_1p^1_1+a_2p^1_2=-a_1\gamma_2+a_2\gamma_1$, then set $h_{(\alpha_1,0)}=a_1q^1_1+a_2q^1_2=a_1T_1+a_2T_2$. These $h_{(\alpha_1,0)}$, together with $g_1$, are all the minimal generators of $\K$ having $x$-degree at least $d_1$. We will determine the minimal generators of $x$-degree $d_1-1$ in \autoref{exgen0II}.
\end{ex}

\begin{ex}[compare {\cite[Proposition~3.4]{bus09} and {\cite[Theorem~5.4]{cd13b}}}]\label{ex2n}
Assume $n=3$ and $d_1=2$. We may assume that the entries of the first column of $\varphi$ are linearly independent, for the other case was considered in \autoref{exgen0}. Then after row operations, we may write
\[\varphi=\begin{bmatrix}x_0^2&\delta_1\\x_0x_1&\delta_2\\x_1^2&\delta_3\end{bmatrix}.\]
Then the resolution \eqref{eq:res1} is
\[0\to R(-1)\oplus R(-1)\xrightarrow{\xi\tr}R^3\xrightarrow{\varphi_1\tr}R(2)\to\Ext^2_R(F/E,R)\to0\]
where
\[\xi=\begin{bmatrix}-x_1&x_0&0\\0&-x_1&x_0\end{bmatrix}.\]
In particular, $\sigma_1=\sigma_2=1$, $s=2$, and $r=0$. By \autoref{kregen}, the minimal generators of $\K$ besides $g_1$ having $x$-degree $\ge 2$ are $g_2\w_1^{\alpha_1}\w_2^{\alpha_2}$ for $\alpha_1,\alpha_2\in\N$ such that $\alpha_1+\alpha_2\le d_2-2$.

Now let us use \autoref{kregenalg} to compute the precise generating set. Let $\rho^1$ and $\rho^2$ be the syzygy matrices of the matrices obtained by deleting the first and second row of $\xi$, respectively. That is,
\[\rho^1=\begin{bmatrix}1&0\\0&x_0\\0&x_1\end{bmatrix}\text{ and }\rho^2=\begin{bmatrix}0&x_0\\0&x_1\\1&0\end{bmatrix}.\]
Therefore
\[\begin{bmatrix}p^1_1&p^1_2\end{bmatrix}=\begin{bmatrix}-x_1&x_0&0\end{bmatrix}\rho^1=\begin{bmatrix}-x_1&x_0^2\end{bmatrix}\]
\[\begin{bmatrix}p^2_1&p^2_2\end{bmatrix}=\begin{bmatrix}0&-x_1&x_0\end{bmatrix}\rho^2=\begin{bmatrix}x_0&-x_1^2\end{bmatrix}\]
\[\begin{bmatrix}q^1_1&q^1_2\end{bmatrix}=\begin{bmatrix}T_1&T_2&T_3\end{bmatrix}\rho^1=\begin{bmatrix}T_1&x_0T_2+x_1T_3\end{bmatrix}\]
\[\begin{bmatrix}q^2_1&q^2_2\end{bmatrix}=\begin{bmatrix}T_1&T_2&T_3\end{bmatrix}\rho^2=\begin{bmatrix}T_3&x_0T_1+x_1T_2\end{bmatrix}.\]
Now we are ready to apply the algorithm from \autoref{kregenalg}. Let $h_{(0,0)}=g_2$. Suppose we have computed $h_{(\alpha_1,\alpha_2)}$ for some $\alpha_1,\alpha_2$ with $\alpha_1+\alpha_2<d_2-2$. Write $h_{(\alpha_1,\alpha_2)}=-a_1x_1+a_2x_0^2=b_1x_0-b_2x_1^2$. (We can do this since $\deg_xh_{(\alpha_1,\alpha_2)}\ge2$.) Set $h_{(\alpha_1+1,\alpha_2)}=a_1T_1+a_2(x_0T_2+x_1T_3)$ and $h_{(\alpha_1,\alpha_2+1)}=b_1T_3+b_2(x_0T_1+x_1T_2)$. By \autoref{kregenalg}, the polynomials $h_{(\alpha_1,\alpha_2)}$ constructed in this way coincide with the generators $g_2\w_1^{\alpha_1}\w_2^{\alpha_2}$.
\end{ex}

\section{The case of plane curves}

Let $n=3$, so that $I=(f_1,f_2,f_3)$ and $\C$ is a plane curve. In this case, we can actually compute a generating set for each slice $\K_{i,*}$ with $i\ge d_1-1$, in \autoref{kregen2}. In particular, we get a generating set of $\K_{\ge d_1-1,*}$ by combining the generators of $\K_{d_1-1,*}$ from \autoref{kregen2} with the minimal generators of $\K$ in $x$-degree $\ge d_1$ from \autoref{kregen}. Unfortunately, the generating set for $\K_{i,*}$ in \autoref{kregen2} is usually not minimal.

Observe that $\K_1=(g_1):\m^\infty=(g_1)$, meaning $\R(E_1)=\Sym(E_1)=S/(g_1)$. Thus to determine $\K$, it suffices to compute $\K/\K_1=\K\R(E_1)=\K_2\R(E_1)$. To accomplish this, we will first determine the $\k[T_1,T_2,T_3]$-module structure of $\R(F)_{i,*}$. This is carried out in the next two lemmas and \autoref{refup}, where we show that $\R(F)_{i,*}$ is generated as a $\k[T_1,T_2,T_3]$-module in $T$-degrees 0 and 1 for all $i\ge-1$.

\begin{lem}\label{mxfac}
Let $\zeta$ be an $n\times(n-1)$ matrix and let $\eta$ be an $(n-1)\times n$ matrix with coefficients in $\m$, such that $\eta\zeta=fI_{n-1}$ for some $f\in R$. Assume that $\grade I_{n-1}(\zeta)=2$. Then $f\in I_{n-1}(\zeta)$.
\end{lem}

\begin{proof}
Let $\chi$ be the $1\times n$ row matrix consisting of the signed minors of $\zeta$. By the Hilbert-Burch theorem, $\im(\chi\tr)=\ker(\zeta\tr)$. By the Hilbert syzygy theorem, $\coker\eta$ has minimal free resolution \[0\to F\xrightarrow{\vartheta}R^n\xrightarrow{\eta}R^{n-1}\to\coker\eta\to 0\] where $F$ is a free $R$-module. We have $f\in\ann_R(\coker\eta)$, so $\rank(\coker\eta)=0$. Using the additivity of rank, we see that $\rank F=1$, hence $\vartheta$ is an $n\times 1$ column matrix.

Now $\eta(\zeta\eta-fI_n)=\eta\zeta\eta-f\eta=(\eta\zeta-fI_{n-1})\eta=0$, so $\im(\zeta\eta-fI_n)\subset\ker\eta=\im\vartheta$. Thus there is a row matrix $\lambda:R^n\to R$ such that $\zeta\eta-fI_n=\vartheta\lambda$. Then $\vartheta(\lambda\zeta)=\zeta\eta\zeta-f\zeta=\zeta(\eta\zeta-fI_{n-1})=0$. Since $\vartheta$ is injective, this means $\lambda\zeta=0$. Hence $\im(\lambda\tr)\subset\ker(\zeta\tr)=\im(\chi\tr)$, so there is $g\in R$ such that $\lambda=g\chi$. We conclude that $\zeta\eta-g\vartheta\chi=fI_n$.

Set \[\zeta'=\begin{bmatrix}\zeta&-g\vartheta\end{bmatrix}\text{ and }\eta'=\begin{bmatrix}\eta\\\chi\end{bmatrix}.\] These are both $n\times n$ matrices, and $\zeta'\eta'=\zeta\eta-g\vartheta\chi=fI_n$. It follows that $\eta'\zeta'=fI_n$. But \[\eta'\zeta'=\begin{bmatrix}\eta\zeta&-g\eta\vartheta\\\chi\zeta&-g\chi\vartheta\end{bmatrix},\] so we get $f=-g\chi\vartheta=\chi(-g\vartheta)$. Therefore $f\in\im\chi=I_{n-1}(\zeta)$.
\end{proof}

In order to compare $\R(E)=\R(E_1)$ and $\R(F)$ in high $T$-degrees, we first compare them in bidegree $(-1,2)$. This will be used as the base case in the proof of \autoref{refup}.

\begin{lem}\label{ref12}
Assume \autoref{reesdata} with $n=3$ and $m=1$. Then $\R(F)_{-1,2}=\R(E)_{0,1}\R(F)_{-1,1}$.
\end{lem}

\begin{proof}
Recall that $\R(F)=R[\w_1,\w_2]$, with $\deg\w_i=(-\sigma_i,1)$. Then
\begin{align*}
\R(F)_{-1,1}&=R_{\sigma_1-1}\w_1\oplus R_{\sigma_2-1}\w_2\\
\R(F)_{-1,2}&=R_{2\sigma_1-1}\w_1^2\oplus R_{\sigma_1+\sigma_2-1}\w_1\w_2\oplus R_{2\sigma_2-1}\w_2^2.
\end{align*}
Because $\dim_{\k}R_\ell=\ell+1$ for any $\ell$, we get $\dim_{\k}\R(F)_{-1,1}=\sigma_1+\sigma_2=d_1$, while $\dim_{\k}\R(F)_{-1,2}=3\sigma_1+3\sigma_2=3d_1$. Set $\mathbf{T}=[T_1\ T_2\ T_3]$, considered as a function $\mathbf{T}:\R(F)_{-1,1}^3\to\R(F)_{-1,2}$. Since $\R(E)_{0,1}=\k T_1+\k T_2+\k T_3$, we have $\R(E)_{0,1}\R(F)_{-1,1}=\im(\mathbf{T})$. It will suffice to show that $\mathbf{T}$ is injective, for then $\dim_{\k}(\R(E)_{0,1}\R(F)_{-1,1})=3(\dim_{\k}\R(F)_{-1,1})=3d_1=\dim_{\k}\R(F)_{-1,2}$, which implies the equality $\R(E)_{0,1}\R(F)_{-1,1}=\R(F)_{-1,2}$.

Suppose $[p_1\ p_2\ p_3]\tr\in\ker(\mathbf{T})$, with $p_i\in \R(F)_{-1,1}$. From \autoref{fpoly}, $[\w_1\ \w_2]\xi=[T_1\ T_2\ T_3]$. Thus
\[0=\mathbf{T}\begin{bmatrix}p_1\\p_2\\p_3\end{bmatrix}=\begin{bmatrix}\w_1&\w_2\end{bmatrix}\xi\begin{bmatrix}p_1\\p_2\\p_3\end{bmatrix}\]
which we can rewrite as \[\begin{bmatrix}p_1&p_2&p_3\end{bmatrix}\xi\tr\begin{bmatrix}\w_1\\\w_2\end{bmatrix}=0.\]
Since $\w_1,\w_2$ is a regular sequence in $\R(F)$, there is $f\in\R(F)_{d_1-1,0}=R_{d_1-1}$ such that
\begin{equation}\label{eq:pxigv}\begin{bmatrix}p_1&p_2&p_3\end{bmatrix}\xi\tr=\begin{bmatrix}-f\w_2&f\w_1\end{bmatrix}.\end{equation}
For $a,b\in\k$, denote by $p_i(a,b)\in R$ the evaluation of $p_i$ at $(\w_1,\w_2)=(a,b)$. Let \[\eta=\begin{bmatrix}p_1(0,-1)&p_2(0,-1)&p_3(0,-1)\\p_1(1,0)&p_2(1,0)&p_3(1,0)\end{bmatrix}.\]
Then \eqref{eq:pxigv} gives \[\eta\xi\tr=\begin{bmatrix}-f\cdot(-1)&f\cdot0\\-f\cdot0&f\cdot1\end{bmatrix}=fI_2.\] By \autoref{mxfac}, $f\in I_2(\xi\tr)$.

Recall the exact sequence \eqref{eq:res1}, which in this setting is \[0\to R(-\sigma_1)\oplus R(-\sigma_2)\xrightarrow{\xi\tr}R^3\xrightarrow{\varphi_m\tr}R(d_1)\to\Ext^2_R(F/E,R)\to0.\] By the Hilbert-Burch theorem, $I_2(\xi\tr)=I_1(\varphi_m\tr)$, therefore $f\in I_2(\xi\tr)=I_1(\varphi_m)$. But $f\in R_{d_1-1}$, while $I_1(\varphi_m)$ is generated in degree $d_1$, therefore $f=0$. Hence by \eqref{eq:pxigv}, $[p_1\ p_2\ p_3]\xi\tr=0$, so $[p_1\ p_2\ p_3]\tr\in\ker\xi$. Recall from \eqref{eq:res2} that $\ker\xi=\im\varphi_m$, which is generated in degree $d_1$. However, $\deg_xp_i\le\sigma_2-1<d_1$, so we must have $p_i=0$, hence $\mathbf{T}$ is injective as claimed.
\end{proof}

The next proposition shows that $\R(M)$ is generated as a module over $\R(E)_{0,*}=\k[T_1,T_2,T_3]$ in $T$-degrees 0 and 1.

\begin{prop}\label{refup}
Assume \autoref{reesdata} with $n=3$ and $m=1$. Let $U=\R(E)_{0,*}=\k[T_1,T_2,T_3]$. Then $\R(F)_{i,*}=U\R(F)_{i,0}+U\R(F)_{i,1}$ for all $i\ge-1$.
\end{prop}

\begin{proof}
We first show by induction on $j$ that \[\R(F)_{-1,j}=U_j\R(F)_{-1,0}+U_{j-1}\R(F)_{-1,1}.\] For $j=0$ and $j=1$ this is clear. For $j\ge 2$,
\begin{align*}
\R(F)_{-1,j}&=\R(F)_{-1,2}\R(F)_{0,j-2}\\
&=U_1\R(F)_{-1,1}\R(F)_{0,j-2}\text{ by \autoref{ref12}}\\
&=U_1\R(F)_{-1,j-1}\\
&=U_1(U_{j-1}\R(F)_{-1,0}+U_{j-2}\R(F)_{-1,1})\text{ by induction}\\
&=U_j\R(F)_{-1,0}+U_{j-1}\R(F)_{-1,1}.
\end{align*}
This proves that $\R(F)_{-1,*}=U\R(F)_{-1,0}+U\R(F)_{-1,1}$. Multiplying both sides by $R_{i+1}=\R(F)_{i+1,0}$ gives $\R(F)_{i,*}=U\R(F)_{i,0}+U\R(F)_{i,1}$.
\end{proof}

By \autoref{refup}, the $U$-module $\R(M)$ is generated in $T$-degrees 0 and 1. In $T$-degree 0, $\R(M)_{*,0}$ is just the ring $R=\k[x_0,x_1]$. To understand $\R(M)_{*,1}$ as a $U$-module, recall that $\R(E)_{*,1}=E$ and $\R(M)_{*,1}=M=F_{\ge0}$. Therefore $\R(M)_{*,1}/\R(E)_{*,1}=(F/E)_{\ge0}$. The free resolution of $F/E$ was given in \autoref{fmode}(a):
\begin{equation}\label{eq:res2b}0\to R(-d_1)\xrightarrow{\varphi_1}R^3\xrightarrow{\xi}R(\sigma_1)\oplus R(\sigma_2)\to F/E\to0.\end{equation}

\begin{rmk}\label{fehf}
Assume \autoref{reesdata} with $n=3$ and $m=1$. Let $H_{F/E}(i)$ be the Hilbert function of $F/E$. Then for all $-1\le i\le d_1-1$, we have $H_{F/E}(i)=d_1-i-1$.
\end{rmk}

\begin{proof}
Recall from \autoref{fmode}(e) that $\sigma_1+\sigma_2=d_1$. Then this is a straightforward computation from \eqref{eq:res2b}.
\end{proof}

For $-1\le i\le d_1-1$, we have the following refinement of \autoref{refup}.

\begin{prop}\label{rffree}
Assume \autoref{reesdata} with $n=3$ and $m=1$. Let $U=\R(E)_{0,*}=\k[T_1,T_2,T_3]$. Then for all $-1\le i\le d_1-1$, there is an isomorphism $\R(F)_{i,*}\cong U^{i+1}\oplus U^{d_1-i-1}(-1)$ under which $\R(E)_{i,*}\cong U^{i+1}$.
\end{prop}

\begin{proof}
The vector space $\R(F)_{i,0}$ is generated by $x_0^i,x_0^{i-1}x_1,\ldots,x_1^i$, so it has dimension $i+1$. By \autoref{fehf}, $\dim_{\k}(F_i/E_i)=d_1-i-1$. Thus there are elements $p_{i,1},\ldots,p_{i,d_1-i-1}\in F_i=\R(F)_{i,1}$ whose images in $F_i/E_i$ form a basis. Since $E_i=\R(E)_{i,1}=U_1\R(E)_{i,0}=U_1\R(F)_{i,0}$, the vector space $\R(F)_{i,1}/U_1\R(F)_{i,0}$ is generated by the images of $p_1,\ldots,p_{d_1-i-1}$. Now consider the map $U^{i+1}\oplus U^{d_1-i-1}(-1)\to \R(F)_{i,*}$ given by $[x_0^i\ \cdots\ x_1^i\ p_1\ \cdots\ p_{d_1-i-1}]$. Then the image of this map contains $\R(F)_{i,0}$, and it contains $\R(F)_{i,1}$. Since $\R(F)_{i,*}=U\R(F)_{i,0}+U\R(F)_{i,1}$ by \autoref{refup}, it is a surjective map, and it sends $U^{i+1}$ to $\R(E)_{i,*}=Ux_0^i+\cdots+Ux_1^i$. To show that it is actually an isomorphism $U^{i+1}\oplus U^{d_1-i-1}(-1)\cong\R(F)_{i,*}$, it suffices to show that the $U$-modules $U^{i+1}\oplus U^{d_1-i-1}(-1)$ and $\R(F)_{i,*}$ have the same Hilbert function.

The Hilbert function of the former is
\begin{align*}
H_{U^{i+1}\oplus U^{d_1-i-1}(-1)}(j)
&=(i+1)\binom{j+2}{2}+(d_1-i-1)\binom{j+1}{2}\\
&=(i+1)(j+1)+d_1\binom{j+1}{2}.
\end{align*}
To compute the Hilbert function of $\R(F)_{i,*}$, note that a basis for $\R(F)_{i,j}$ consists of monomials $x_0^kx_1^\ell\w_1^{j-a}\w_2^a$ where $k+\ell=(j-a)\sigma_1+a\sigma_2+i$. Thus
\begin{align*}
H_{\R(F)_{i,*}}(j)
&=\sum_{a=0}^j((j-a)\sigma_1+a\sigma_2+i+1)\\
&=\sigma_1\sum_{a=0}^j(j-a)+\sigma_2\sum_{a=0}^ja+(i+1)(j+1)\\
&=\sigma_1\binom{j+1}{2}+\sigma_2\binom{j+1}{2}+(i+1)(j+1)\\
&=(i+1)(j+1)+d_1\binom{j+1}{2}\\
&=H_{U^{i+1}\oplus U^{d_1-i-1}(-1)}(j).
\end{align*}
Therefore $\R(F)_{i,*}\cong U^{i+1}\oplus U^{d_1-i-1}(-1)$.
\end{proof}

We may now compute the generators of $\K_{i,*}$ for any $i\ge d_1-1$.

\begin{thm}\label{kregen2}
Assume \autoref{reesdata} with $n=3$ and $m=1$. Fix $i\ge d_1-1$ and set $c=d_2-i$. Let $U=\k[T_1,T_2,T_3]$. Let $\Omega_c$ and $\mathcal{B}_c$ be as defined in \autoref{dfrfgen}. For each $0\le\ell\le d_1-2$, choose elements $p_{\ell,1},\ldots,p_{\ell,d_1-\ell-1}\in F_\ell=\R(F)_{\ell,1}$ whose images in $F_\ell/E_\ell$ form a basis. Then \[\{g_1x_0^{i-d_1},\ldots,g_1x_1^{i-d_1}\}\cup g_2\mathcal{B}_c\cup\{g_2p_{\as-c,j}\ut^{\ua}\mid\alpha\in\Omega_c, 1\le j\le d_1-\as+c-1\}\] is a generating set (not necessarily minimal) for $\K_{i,*}$ as a $U$-module.
\end{thm}

\begin{proof}
Since $\R(E)=\R(E_1)\cong S/(g_1)$, and $g_1$ has $x$-degree $d_1$, we just need to show that
\[g_2\mathcal{B}_c\cup\{g_2p_{\as-c,j}\ut^{\ua}\mid\alpha\in\Omega_c, 1\le j\le d_1-\as+c-1\}\]
is a minimal generating set for $\K_{i,*}\R(E)$. We showed in \eqref{eq:jrf} that $\K_{i,*}\R(E)=g_2\R(F)_{-c,*}$. From \autoref{rfgen}(b), we know that $g_2\mathcal{B}_c$ is a minimal generating set for $\K_{i,*}\R(E)$ as an $\R(M)_{0,*}$-module. Recalling the definition of $\mathcal{B}_c$, this means that
\begin{equation}\label{eq:kupgen}\K_{i,*}\R(E)=g_2\sum_{\ua\in\Omega_c}\R(M)_{0,*}(x_0,x_1)^{\as-c}\ut^{\ua}=g_2\sum_{\ua\in\Omega_c}\R(M)_{\as-c,*}\ut^{\ua}.\end{equation}
By \autoref{refup}, $\R(M)_{\as-c,*}=U\R(M)_{\as-c,0}+U\R(M)_{\as-c,1}$. The vector space $\R(M)_{\as-c,0}$ is spanned by $\{x_0^{\as-c-j}x_1^j\mid 0\le j\le\as-c\}$, while the vector space $\R(M)_{\as-c,1}/U_1\R(M)_{\as-c,0}=(F/E)_{\as-c}$ is spanned by the images of $\{p_{\as-c,j}\mid 1\le j\le d_1-\as+c-1\}$. Therefore $\R(M)_{\as-c,*}$ is generated as an $U$-module by the union of these two sets. Then by \eqref{eq:kupgen}, $\K_{i,*}\R(E)$ is generated as an $U$-module by the union of \[g_2\mathcal{B}_c=\{g_2x_0^{\as-c-j}x_1^j\ut^{\ua}\mid\ua\in\Omega_c,\ 0\le j\le\as-c\}\] and \[\{g_2p_{\as-c,j}\ut^{\ua}\mid\alpha\in\Omega_c,\ 1\le j\le d_1-\as+c-1\}.\qedhere\]
\end{proof}

\begin{ex}[continuation of \autoref{exgen0}]\label{exgen0II}
Assume $n=3$ and assume that the entries of the first column of $\varphi$ are linearly dependent. In \autoref{exgen0}, we computed the minimal generators of $\K$ with $x$-degree $\ge d_1$. As for the generators having $x$-degree $d_1-1$, we claim that the generating set given in \autoref{kregen2} is minimal. Note that $\Omega_c$ (where $c=d_2-d_1+1$) in this case consists of only one element, namely $(\alpha_1,0)$ where $\alpha_1=\lceil\frac{d_2}{d_1}\rceil$ (as in the $\sigma_2=0$ case of \autoref{exalp}). Therefore the generating set in \autoref{kregen2} is just \[g_2\w_1^{\alpha_1}\{x_0^{\alpha_1\sigma_1-c},\ldots,x_1^{\alpha_1\sigma_1-c},p_{\alpha_1\sigma_1-c,1},\ldots,p_{\alpha_1\sigma_1-c,d_1-\alpha_1\sigma_1+c-1}\}.\]
By \autoref{rffree}, $\{x_0^{\alpha_1\sigma_1-c},\ldots,x_1^{\alpha_1\sigma_1-c},p_{\alpha_1\sigma_1-c,1},\ldots,p_{\alpha_1\sigma_1-c,d_1-\alpha_1\sigma_1+c-1}\}$ is a minimal generating set for $\R(F)_{\alpha_1\sigma_1-c,*}$ as a $U$-module. Thus the generating set from \autoref{kregen2} is a minimal generating set for $\K$ in $x$-degree $d_1-1$.
\end{ex}

\begin{ex}[continuation of \autoref{ex2n}]\label{ex2nII}
Assume $n=3$ and $d_1=2$, and assume that the entries of the first column of $\varphi$ are linearly independent, as in \autoref{ex2n}. To determine the minimal generators of $\K$ having $x$-degree 1, first note that $\Omega_{d_2-1}=\{(\alpha_1,\alpha_2)\mid\alpha_1+\alpha_2=d_2-1\}$, therefore $g_2\mathcal{B}_{d_2-1}=\{g_2\w_1^{\alpha_1}\w_2^{\alpha_2}\mid\alpha_1+\alpha_2=d_2-1\}$. We claim that this is a minimal generating set of $\K_{1,*}$. Minimality will follow from \autoref{rfgen}, so we just have to show it is generating. By \autoref{kregen2}, the only other generators of $\K_{1,*}$ have the form $g_2p_{\alpha_1+\alpha_2-d_2+1,j}\w_1^{\alpha_1}\w_2^{\alpha_2}$ for $(\alpha_1,\alpha_2)\in\Omega_{d_2-1}$. But since $(\alpha_1,\alpha_2)\in\Omega_{d_2-1}$, this means $\alpha_1+\alpha_2-d_2+1=0$ so the generators have the form $g_2p_{0,1}\w_1^{\alpha_1}\w_2^{\alpha_2}$, where $p_{0,1}$ is a lift of a generator of $F_1/E_1$. Since $\alpha_1+\alpha_2=d_2-1$,
\[g_2p_{0,1}\w_1^{\alpha_1}\w_2^{\alpha_2}\in g_2\R(F)_{0,1}(\w_1,\w_2)^{d_2-1}\subset g_2\R(F)_{-1,2}(\w_1,\w_2)^{d_2-2}.\]
But $\R(F)_{-1,2}=U_1\R(F)_{-1,1}$ by \autoref{refup}, and $\R(F)_{-1,1}$ is generated by $\w_1$ and $\w_2$, thus
\[g_2p_{0,1}\w_1^{\alpha_1}\w_2^{\alpha_2}\in g_2\R(F)_{-1,2}(\w_1,\w_2)^{d_2-2}\subset U_1(g_2(\w_1,\w_2)^{d_2-1}).\]
Therefore $g_2p_{0,1}\w_1^{\alpha_1}\w_2^{\alpha_2}$ is in the $U$-submodule generated by $g_2\mathcal{B}_{d_2-1}$. Hence $g_2p_{0,1}\w_1^{\alpha_1}\w_2^{\alpha_2}$ is not part of a minimal generating set, so the only minimal generators are $g_2\mathcal{B}_{d_2-1}$.

Combining this with what we already showed in \autoref{ex2n}, we conclude that all the minimal generators of $\K$ besides $g_1$ having $x$-degree $\ge1$ are $g_2\w_1^{\alpha_1}\w_2^{\alpha_2}$ for $\alpha_1,\alpha_2\in\N$ such that $\alpha_1+\alpha_2\le d_2-1$. Since the only generator of $\K$ in degree 0 is the equation of the fiber ring, in bidegree $(0,d)$, we have all the minimal generators of $\K$.
\end{ex}

\begin{ex}
Unlike in the previous examples, it is not always possible to determine which, or even how many, of the generators in \autoref{kregen2} are minimal, from only the numerical data of $d_1$, $d_2$, $\sigma_1$, and $\sigma_2$. For example, take
\[\varphi=\begin{bmatrix}x_0^4&x_1^7\\x_0^2x_1^2&0\\x_1^4&x_0^7\end{bmatrix}.\]
Then $d_1=4$, $d_2=7$, and $\sigma_1=\sigma_2=2$. A computation in Macaulay2 shows that $\K_{3,*}$ has seven minimal generators: three in degree 3, and four in degree 4. On the other hand, consider
\[\varphi=\begin{bmatrix}x_0^4+x_0^3x_1&x_1^7\\x_0^2x_1^2&0\\x_1^4&x_0^7\end{bmatrix}.\]
Then, as before, we have $d_1=4$, $d_2=7$, and $\sigma_1=\sigma_2=2$. But in this case, $\K_{3,*}$ has only six minimal generators: three in degree 3, and three in degree 4.
\end{ex}

\section*{Acknowledgments}

This work was carried out as part of the author's doctoral thesis research. The author is very grateful to his advisor, Claudia Polini, for suggesting the problem and for helpful discussions, as well as her comments on an early version of the paper.

\bibliographystyle{abbrv}
\bibliography{rees-sing}
\end{document}